\documentclass[a4paper,reqno,12pt,oneside]{amsart}
\usepackage{amsmath}
\usepackage{graphicx}
\usepackage{mathrsfs}
\usepackage{amssymb,amsmath, amsfonts, amsthm}
\usepackage{latexsym}
\usepackage{color} 
\usepackage[dvips]{epsfig}
\usepackage[colorlinks]{hyperref}
\usepackage{tikz}
\usepackage{pgfplots}
\usepackage[hmargin=2.8cm, vmargin=2.8cm]{geometry}
 \allowdisplaybreaks

\numberwithin{equation}{section}

\newtheorem{thm}{Theorem}[section]

\newtheorem{defi}[thm]{Definition}
\newtheorem{prop}[thm]{Proposition}
\newtheorem{lm}[thm]{Lemma}
\newtheorem{rem}[thm]{Remark}

\newcommand\lap{\Delta}

\newcommand\be{\beta}
\newcommand\ga{\gamma}
\newcommand\al{\alpha}

\title[Entire solutions in a strongly competitive regime]{Entire solutions to a strongly competitive nonlinear Schrödinger system}

\author{Pierpaolo Esposito}
\address{Pierpaolo Esposito, Dipartimento di Matematica e Fisica, Universit\`a degli Studi Roma Tre, Largo S.~Leonardo Murialdo 1, Roma 00146, Italy.}
\email{pierpaolo.esposito@uniroma3.it}

\author{Pablo Figueroa}
\address{Pablo Figueroa, Instituto de Ciencias Físicas y Matemáticas, Facultad de  Ciencias, Universidad Austral de Chile, Campus Isla Teja s/n, Valdivia, Chile}
\email{pablo.figueroa@uach.cl}

\author{Angela Pistoia}
\address{Angela Pistoia, Dipartimento SBAI, Universit\`a di Roma ``La Sapienza'',  via Antonio Scarpa 16, 00161 Roma, Italy}
\email{angela.pistoia@uniroma1.it}

\author{Giusi Vaira}
\address{Giusi Vaira,
Dipartimento di Matematica, Universit\`a degli studi di Bari ``Aldo Moro'', via Edoardo Orabona 4,70125 Bari, Italy}
\email{giusi.vaira@uniba.it}

\begin{document}

\begin{abstract} 
%
%
%
We build infinitely-many non-radial positive solutions to the Schrödinger system 
\begin{equation*} 
\left\{\begin{aligned}
&-\Delta u_1+u_1=u_1^{{\mathfrak p} }-\Lambda u_1^{a_1} u_2^{a_2}\ \hbox{in}\ \mathbb R^N\\
&-\Delta u_2+u_2=u_2^{{\mathfrak p} }-\Lambda u_1^{b_1}u_2^{b_2} \ \hbox{in}\ \mathbb R^N\\
\end{aligned}\right.
\end{equation*}
with sub-critical $\mathfrak p$-growth as $\Lambda \to +\infty$. The profile of each component is the sum of several copies of the positive solution to $-\Delta U+U=U^{{\mathfrak p} }$ in $\mathbb R^N$, 
centered at suitable {\em peaks} whose mutual distances diverge as $\Lambda$ increases.  More precisely,  given two concentric regular polygons with $k$ sides and very large radii,  the peaks of the first component are arranged  
along the edges of the {\em outer} polygon, alternated with those of the second component, and along the $k$ rays joining the vertices of the  two polygons. 
To the best of our knowledge, this provides the first example of non-radial positive solutions for strongly competitive Schr\"odinger systems in the whole space.\end{abstract}

\maketitle

\section{Introduction}

In this paper we investigate existence issues for the non-linear Schrödinger system
\begin{equation} \label{ests}
\left\{\begin{aligned}
&-\Delta u_1+u_1=u_1^{{\mathfrak p} }-\Lambda u_1^{a_1} u_2^{a_2}\ \hbox{in}\ \mathbb R^N\\
&-\Delta u_2+u_2=u_2^{{\mathfrak p} }-\Lambda u_1^{b_1}u_2^{b_2} \ \hbox{in}\ \mathbb R^N\\
\end{aligned}\right.
\end{equation}
with $N\ge2$, ${\mathfrak p}>1$, $a_i,b_i>0$ and $\Lambda\in\mathbb R.$
\\

Except for some very special choices of the exponents $a_i$ and $b_i$, system \eqref{ests} does not possess a variational structure. While an extensive literature is available for variational Schr\"odinger systems, both in bounded domains and in the whole space, much less is known in the genuinely non-variational setting.

A particularly relevant case arises when  ${\mathfrak p}=3$ and $a_1=b_2=a_2-1=b_1-1=1$ in which case \eqref{ests} reduces to 
\begin{equation} \label{bose}
\left\{\begin{aligned}
&-\Delta u_1+u_1=u_1^{3}-\Lambda u _1u_2^{ 2}\ \hbox{in}\ \mathbb R^N\\
&-\Delta u_2+u_2=u_2^{3}-\Lambda u_1^{2}u_2  \ \hbox{in}\ \mathbb R^N,\\
\end{aligned}\right.
\end{equation}
which arises  in the study of  Bose-Einstein condensates 
 This system has been widely studied over the last decades, especially on bounded domains or in the presence of external potentials, where the linear terms are replaced by non-autonomous ones of the form $V_i(x)u_i$. We refer to the recent paper \cite{li-wei-wu} by Li, Wei and Wu for an exhaustive list of references.
\\
An interesting phenomenon emerges in the strongly competitive regime $\Lambda \to +\infty$: the two components tend to segregate, meaning that their supports become asymptotically disjoint. This phenomenon, known as \emph{phase separation}, was first investigated by Conti, Terracini and Verzini \cite{conti-terracini-verzini2,conti-terracini-verzini} in the context of least-energy solutions. In particular, Wei and Weth \cite{wei-weth} proved that system \eqref{bose} admits a radial solution $(u_1,u_2)$ such that, as $\Lambda \to +\infty$, the difference $u_1-u_2$ converges to a sign-changing radial solution $W$ of the limit equation
$$- \Delta W + W = W^3 \quad \text{in } \mathbb{R}^N.$$
More precisely, the two components converge to the positive and negative parts of $W$, respectively. This result was extended by Terracini and Verzini \cite{terracini-verzini} to systems with $k$ components, where each component segregates and converges to one nodal region of a radial solution with $k$ nodal domains, as constructed by Bartsch and Willem  \cite{bartsch-willem}.

Motivated by this close connection between positive solutions of competitive systems and sign-changing solutions of the corresponding scalar equation, a natural question arises.

\medskip
\noindent
\textbf{(Q1)} \emph{Given a sign-changing solution $W$ of}
\begin{equation}
\label{lim-nod-2}
- \Delta W + W = |W|^{p-1} W \quad \text{in } \mathbb{R}^N,
\end{equation}
\emph{does there exist a positive solution $(u_1,u_2)$ of \eqref{ests} such that $u_1$ and $u_2$ resemble $W^+$ and $W^-$ as $\Lambda \to +\infty$?}

\medskip
In the radial case, this question has been answered positively in  \cite{terracini-verzini,wei-weth}. In the non-radial setting, however, the problem remains widely open and appears to be extremely challenging. In the present work we address a more flexible version of this question.

\medskip
\noindent
\textbf{(Q2)} \emph{Given a sign-changing approximate solution $W^\ast$ of \eqref{lim-nod-2}, can one find a positive solution $(u_1,u_2)$ of \eqref{ests} such that $u_1$ and $u_2$ resemble $(W^\ast)^+$ and $(W^\ast)^-$ as $\Lambda \to +\infty$?}

\medskip
The aim of this paper is to provide a positive answer to \textbf{(Q2)} when $W^\ast$ is the approximate solution introduced by Pacard, Musso and Wei \cite{musso-pacard-wei} in their construction of non-radial sign-changing solutions to \eqref{lim-nod-2}. These solutions are invariant under rotations of angle $2\pi/k$ in the $(x_1,x_2)$--plane for some $k \ge 7$.\\
Let us  be more precise.
Given $U \in H^1(\mathbb{R}^N)$ the unique positive radial solution of
\begin{equation}\label{bubble}-\Delta U+U=U^{{\mathfrak p}} \qquad \hbox{in }\mathbb{R}^N,\end{equation}
via the correspondence $W^+ \leftrightarrow u_1$ and $W^-\leftrightarrow u_2$, the approximate solution $W^*$ to \eqref{lim-nod-2} in \cite{musso-pacard-wei} reads as an approximate solution $(U_1^*,U_2^*)$ to \eqref{ests}, where $U_1^*$ and $U_2^*$ (defined in \eqref{uistar}) are sum of several (positive) copies of $U$ centered at the positive and negative peaks of $W^*$ as in Figure \ref{figure}.

\tikzset{every picture/.style={line width=0.75pt}} 

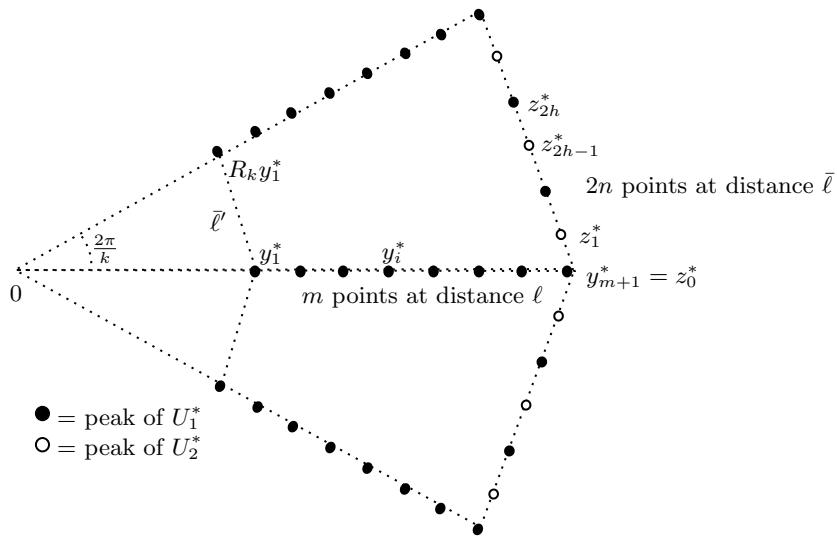
\begin{figure}[ht]
    \centering
 
\begin{tikzpicture}[x=0.75pt,y=0.75pt,yscale=-1,xscale=1]\label{figure}

\draw  [dash pattern={on 0.84pt off 2.51pt}] (351.09,14.67) -- (398.12,145.47) -- (118.81,144.98) -- cycle ;
\draw  [dash pattern={on 0.84pt off 2.51pt}] (351.13,274.25) -- (398.13,144.48) -- (118.81,144.97) -- cycle ;
\draw  [fill={rgb, 255:red, 0; green, 0; blue, 0 }  ,fill opacity=1 ] (236.42,145.69) .. controls (236.42,144.3) and (237.36,143.18) .. (238.53,143.18) .. controls (239.7,143.18) and (240.65,144.3) .. (240.65,145.69) .. controls (240.65,147.08) and (239.7,148.2) .. (238.53,148.2) .. controls (237.36,148.2) and (236.42,147.08) .. (236.42,145.69) -- cycle ;
\draw  [fill={rgb, 255:red, 0; green, 0; blue, 0 }  ,fill opacity=1 ] (259.11,145.69) .. controls (259.11,144.3) and (260.06,143.18) .. (261.22,143.18) .. controls (262.39,143.18) and (263.34,144.3) .. (263.34,145.69) .. controls (263.34,147.08) and (262.39,148.2) .. (261.22,148.2) .. controls (260.06,148.2) and (259.11,147.08) .. (259.11,145.69) -- cycle ;
\draw  [fill={rgb, 255:red, 0; green, 0; blue, 0 }  ,fill opacity=1 ] (280.43,145.69) .. controls (280.43,144.3) and (281.37,143.18) .. (282.54,143.18) .. controls (283.71,143.18) and (284.66,144.3) .. (284.66,145.69) .. controls (284.66,147.08) and (283.71,148.2) .. (282.54,148.2) .. controls (281.37,148.2) and (280.43,147.08) .. (280.43,145.69) -- cycle ;
\draw  [fill={rgb, 255:red, 0; green, 0; blue, 0 }  ,fill opacity=1 ] (303.12,145.69) .. controls (303.12,144.3) and (304.07,143.18) .. (305.23,143.18) .. controls (306.4,143.18) and (307.35,144.3) .. (307.35,145.69) .. controls (307.35,147.08) and (306.4,148.2) .. (305.23,148.2) .. controls (304.07,148.2) and (303.12,147.08) .. (303.12,145.69) -- cycle ;

\draw  [fill={rgb, 255:red, 0; green, 0; blue, 0 }  ,fill opacity=1 ] (325.54,145.69) .. controls (325.54,144.3) and (326.48,143.18) .. (327.65,143.18) .. controls (328.82,143.18) and (329.77,144.3) .. (329.77,145.69) .. controls (329.77,147.08) and (328.82,148.2) .. (327.65,148.2) .. controls (326.48,148.2) and (325.54,147.08) .. (325.54,145.69) -- cycle ;
\draw  [fill={rgb, 255:red, 0; green, 0; blue, 0 }  ,fill opacity=1 ] (348.23,145.69) .. controls (348.23,144.3) and (349.18,143.18) .. (350.34,143.18) .. controls (351.51,143.18) and (352.46,144.3) .. (352.46,145.69) .. controls (352.46,147.08) and (351.51,148.2) .. (350.34,148.2) .. controls (349.18,148.2) and (348.23,147.08) .. (348.23,145.69) -- cycle ;
\draw  [fill={rgb, 255:red, 0; green, 0; blue, 0 }  ,fill opacity=1 ] (369.55,145.69) .. controls (369.55,144.3) and (370.49,143.18) .. (371.66,143.18) .. controls (372.83,143.18) and (373.78,144.3) .. (373.78,145.69) .. controls (373.78,147.08) and (372.83,148.2) .. (371.66,148.2) .. controls (370.49,148.2) and (369.55,147.08) .. (369.55,145.69) -- cycle ;
\draw  [fill={rgb, 255:red, 0; green, 0; blue, 0 }  ,fill opacity=1 ] (392.24,145.69) .. controls (392.24,144.3) and (393.19,143.18) .. (394.36,143.18) .. controls (395.52,143.18) and (396.47,144.3) .. (396.47,145.69) .. controls (396.47,147.08) and (395.52,148.2) .. (394.36,148.2) .. controls (393.19,148.2) and (392.24,147.08) .. (392.24,145.69) -- cycle ;

\draw  [fill={rgb, 255:red, 255; green, 255; blue, 255 }  ,fill opacity=1 ] (358.59,35.26) .. controls (359.7,34.84) and (360.94,35.47) .. (361.36,36.66) .. controls (361.78,37.85) and (361.23,39.15) .. (360.12,39.57) .. controls (359.01,39.99) and (357.77,39.36) .. (357.34,38.17) .. controls (356.92,36.98) and (357.48,35.68) .. (358.59,35.26) -- cycle ;
\draw  [fill={rgb, 255:red, 0; green, 0; blue, 0 }  ,fill opacity=1 ] (366.82,58.39) .. controls (367.93,57.97) and (369.17,58.6) .. (369.59,59.79) .. controls (370.02,60.98) and (369.46,62.28) .. (368.35,62.7) .. controls (367.24,63.11) and (366,62.49) .. (365.57,61.3) .. controls (365.15,60.11) and (365.71,58.8) .. (366.82,58.39) -- cycle ;
\draw  [fill={rgb, 255:red, 255; green, 255; blue, 255 }  ,fill opacity=1 ] (374.55,80.11) .. controls (375.66,79.69) and (376.9,80.32) .. (377.33,81.51) .. controls (377.75,82.7) and (377.19,84) .. (376.08,84.42) .. controls (374.97,84.84) and (373.73,84.21) .. (373.31,83.02) .. controls (372.88,81.83) and (373.44,80.53) .. (374.55,80.11) -- cycle ;
\draw  [fill={rgb, 255:red, 0; green, 0; blue, 0 }  ,fill opacity=1 ] (382.78,103.24) .. controls (383.89,102.82) and (385.14,103.45) .. (385.56,104.64) .. controls (385.98,105.83) and (385.43,107.13) .. (384.32,107.55) .. controls (383.21,107.96) and (381.96,107.34) .. (381.54,106.15) .. controls (381.12,104.96) and (381.67,103.66) .. (382.78,103.24) -- cycle ;
\draw  [fill={rgb, 255:red, 255; green, 255; blue, 255 }  ,fill opacity=1 ] (390.52,124.96) .. controls (391.63,124.55) and (392.87,125.17) .. (393.29,126.36) .. controls (393.72,127.55) and (393.16,128.85) .. (392.05,129.27) .. controls (390.94,129.69) and (389.7,129.06) .. (389.27,127.87) .. controls (388.85,126.68) and (389.41,125.38) .. (390.52,124.96) -- cycle ;

\draw  [fill={rgb, 255:red, 0; green, 0; blue, 0 }  ,fill opacity=1 ] (219.36,202.42) .. controls (220.01,201.2) and (221.31,200.65) .. (222.28,201.19) .. controls (223.24,201.73) and (223.5,203.16) .. (222.85,204.38) .. controls (222.21,205.6) and (220.9,206.15) .. (219.94,205.61) .. controls (218.97,205.07) and (218.72,203.64) .. (219.36,202.42) -- cycle ;
\draw  [fill={rgb, 255:red, 0; green, 0; blue, 0 }  ,fill opacity=1 ] (238.11,212.91) .. controls (238.75,211.7) and (240.06,211.14) .. (241.02,211.68) .. controls (241.99,212.22) and (242.24,213.65) .. (241.6,214.87) .. controls (240.95,216.09) and (239.65,216.64) .. (238.68,216.1) .. controls (237.72,215.56) and (237.46,214.13) .. (238.11,212.91) -- cycle ;
\draw  [fill={rgb, 255:red, 0; green, 0; blue, 0 }  ,fill opacity=1 ] (255.71,222.77) .. controls (256.36,221.55) and (257.66,221) .. (258.63,221.54) .. controls (259.59,222.08) and (259.85,223.51) .. (259.2,224.72) .. controls (258.56,225.94) and (257.25,226.49) .. (256.29,225.95) .. controls (255.32,225.41) and (255.07,223.99) .. (255.71,222.77) -- cycle ;
\draw  [fill={rgb, 255:red, 0; green, 0; blue, 0 }  ,fill opacity=1 ] (274.45,233.26) .. controls (275.1,232.04) and (276.41,231.49) .. (277.37,232.03) .. controls (278.33,232.57) and (278.59,234) .. (277.95,235.22) .. controls (277.3,236.43) and (275.99,236.98) .. (275.03,236.44) .. controls (274.07,235.9) and (273.81,234.48) .. (274.45,233.26) -- cycle ;

\draw  [fill={rgb, 255:red, 0; green, 0; blue, 0 }  ,fill opacity=1 ] (292.97,243.62) .. controls (293.62,242.41) and (294.92,241.86) .. (295.89,242.4) .. controls (296.85,242.93) and (297.11,244.36) .. (296.46,245.58) .. controls (295.82,246.8) and (294.51,247.35) .. (293.55,246.81) .. controls (292.58,246.27) and (292.32,244.84) .. (292.97,243.62) -- cycle ;
\draw  [fill={rgb, 255:red, 0; green, 0; blue, 0 }  ,fill opacity=1 ] (311.71,254.12) .. controls (312.36,252.9) and (313.66,252.35) .. (314.63,252.89) .. controls (315.59,253.43) and (315.85,254.85) .. (315.2,256.07) .. controls (314.56,257.29) and (313.25,257.84) .. (312.29,257.3) .. controls (311.32,256.76) and (311.07,255.33) .. (311.71,254.12) -- cycle ;
\draw  [fill={rgb, 255:red, 0; green, 0; blue, 0 }  ,fill opacity=1 ] (329.32,263.97) .. controls (329.97,262.75) and (331.27,262.2) .. (332.23,262.74) .. controls (333.2,263.28) and (333.46,264.71) .. (332.81,265.93) .. controls (332.16,267.14) and (330.86,267.69) .. (329.9,267.16) .. controls (328.93,266.62) and (328.67,265.19) .. (329.32,263.97) -- cycle ;
\draw  [fill={rgb, 255:red, 0; green, 0; blue, 0 }  ,fill opacity=1 ] (348.06,274.46) .. controls (348.71,273.24) and (350.01,272.69) .. (350.98,273.23) .. controls (351.94,273.77) and (352.2,275.2) .. (351.55,276.42) .. controls (350.91,277.64) and (349.6,278.19) .. (348.64,277.65) .. controls (347.67,277.11) and (347.42,275.68) .. (348.06,274.46) -- cycle ;

\draw  [fill={rgb, 255:red, 0; green, 0; blue, 0 }  ,fill opacity=1 ] (218.01,86.37) .. controls (217.4,85.14) and (217.69,83.72) .. (218.67,83.2) .. controls (219.65,82.68) and (220.94,83.27) .. (221.56,84.5) .. controls (222.17,85.74) and (221.88,87.16) .. (220.9,87.67) .. controls (219.92,88.19) and (218.63,87.61) .. (218.01,86.37) -- cycle ;
\draw  [fill={rgb, 255:red, 0; green, 0; blue, 0 }  ,fill opacity=1 ] (237.04,76.33) .. controls (236.43,75.1) and (236.72,73.68) .. (237.7,73.16) .. controls (238.68,72.65) and (239.97,73.23) .. (240.59,74.46) .. controls (241.21,75.7) and (240.91,77.12) .. (239.93,77.63) .. controls (238.96,78.15) and (237.66,77.57) .. (237.04,76.33) -- cycle ;
\draw  [fill={rgb, 255:red, 0; green, 0; blue, 0 }  ,fill opacity=1 ] (254.92,66.9) .. controls (254.31,65.67) and (254.6,64.25) .. (255.58,63.73) .. controls (256.56,63.22) and (257.85,63.8) .. (258.47,65.03) .. controls (259.09,66.27) and (258.79,67.69) .. (257.81,68.21) .. controls (256.83,68.72) and (255.54,68.14) .. (254.92,66.9) -- cycle ;
\draw  [fill={rgb, 255:red, 0; green, 0; blue, 0 }  ,fill opacity=1 ] (273.96,56.87) .. controls (273.34,55.63) and (273.63,54.21) .. (274.61,53.7) .. controls (275.59,53.18) and (276.89,53.76) .. (277.5,55) .. controls (278.12,56.23) and (277.83,57.65) .. (276.85,58.17) .. controls (275.87,58.68) and (274.57,58.1) .. (273.96,56.87) -- cycle ;

\draw  [fill={rgb, 255:red, 0; green, 0; blue, 0 }  ,fill opacity=1 ] (292.76,46.95) .. controls (292.14,45.72) and (292.44,44.3) .. (293.42,43.78) .. controls (294.4,43.26) and (295.69,43.85) .. (296.31,45.08) .. controls (296.92,46.31) and (296.63,47.73) .. (295.65,48.25) .. controls (294.67,48.77) and (293.38,48.19) .. (292.76,46.95) -- cycle ;
\draw  [fill={rgb, 255:red, 0; green, 0; blue, 0 }  ,fill opacity=1 ] (311.79,36.91) .. controls (311.18,35.68) and (311.47,34.26) .. (312.45,33.74) .. controls (313.43,33.23) and (314.72,33.81) .. (315.34,35.04) .. controls (315.96,36.28) and (315.66,37.7) .. (314.68,38.21) .. controls (313.7,38.73) and (312.41,38.15) .. (311.79,36.91) -- cycle ;
\draw  [fill={rgb, 255:red, 0; green, 0; blue, 0 }  ,fill opacity=1 ] (329.67,27.48) .. controls (329.06,26.25) and (329.35,24.83) .. (330.33,24.31) .. controls (331.31,23.8) and (332.6,24.38) .. (333.22,25.61) .. controls (333.83,26.85) and (333.54,28.27) .. (332.56,28.78) .. controls (331.58,29.3) and (330.29,28.72) .. (329.67,27.48) -- cycle ;
\draw  [fill={rgb, 255:red, 0; green, 0; blue, 0 }  ,fill opacity=1 ] (348.7,17.45) .. controls (348.09,16.21) and (348.38,14.79) .. (349.36,14.27) .. controls (350.34,13.76) and (351.63,14.34) .. (352.25,15.58) .. controls (352.87,16.81) and (352.57,18.23) .. (351.59,18.75) .. controls (350.62,19.26) and (349.32,18.68) .. (348.7,17.45) -- cycle ;

\draw  [fill={rgb, 255:red, 255; green, 255; blue, 255 }  ,fill opacity=1 ] (390.86,165.99) .. controls (391.97,166.41) and (392.52,167.71) .. (392.09,168.9) .. controls (391.66,170.09) and (390.42,170.71) .. (389.31,170.28) .. controls (388.2,169.86) and (387.65,168.56) .. (388.08,167.37) .. controls (388.51,166.18) and (389.76,165.56) .. (390.86,165.99) -- cycle ;
\draw  [fill={rgb, 255:red, 0; green, 0; blue, 0 }  ,fill opacity=1 ] (382.51,189.07) .. controls (383.62,189.49) and (384.17,190.79) .. (383.74,191.98) .. controls (383.31,193.17) and (382.06,193.79) .. (380.95,193.37) .. controls (379.85,192.94) and (379.3,191.64) .. (379.73,190.45) .. controls (380.16,189.26) and (381.4,188.64) .. (382.51,189.07) -- cycle ;
\draw  [fill={rgb, 255:red, 255; green, 255; blue, 255 }  ,fill opacity=1 ] (374.66,210.75) .. controls (375.77,211.17) and (376.32,212.48) .. (375.89,213.66) .. controls (375.46,214.85) and (374.21,215.47) .. (373.11,215.05) .. controls (372,214.62) and (371.45,213.32) .. (371.88,212.13) .. controls (372.31,210.94) and (373.56,210.32) .. (374.66,210.75) -- cycle ;
\draw  [fill={rgb, 255:red, 0; green, 0; blue, 0 }  ,fill opacity=1 ] (366.31,233.83) .. controls (367.42,234.25) and (367.97,235.56) .. (367.54,236.74) .. controls (367.11,237.93) and (365.86,238.55) .. (364.75,238.13) .. controls (363.65,237.7) and (363.1,236.4) .. (363.53,235.21) .. controls (363.96,234.02) and (365.2,233.4) .. (366.31,233.83) -- cycle ;
\draw  [fill={rgb, 255:red, 255; green, 255; blue, 255 }  ,fill opacity=1 ] (358.46,255.51) .. controls (359.57,255.93) and (360.12,257.24) .. (359.69,258.43) .. controls (359.26,259.61) and (358.01,260.23) .. (356.91,259.81) .. controls (355.8,259.39) and (355.25,258.08) .. (355.68,256.89) .. controls (356.11,255.71) and (357.35,255.09) .. (358.46,255.51) -- cycle ;

\draw  [dash pattern={on 0.84pt off 2.51pt}]  (220.9,87.67) -- (238.53,143.18) ;
\draw  [dash pattern={on 0.84pt off 2.51pt}]  (238.53,145.69) -- (222.28,201.19) ;
\draw  [draw opacity=0][dash pattern={on 0.84pt off 2.51pt}] (152.08,128.05) .. controls (155.11,132.22) and (156.86,137.17) .. (156.86,142.47) .. controls (156.86,143.05) and (156.84,143.62) .. (156.8,144.18) -- (125.9,142.47) -- cycle ; \draw  [dash pattern={on 0.84pt off 2.51pt}] (152.08,128.05) .. controls (155.11,132.22) and (156.86,137.17) .. (156.86,142.47) .. controls (156.86,143.05) and (156.84,143.62) .. (156.8,144.18) ;  
\draw  [fill={rgb, 255:red, 0; green, 0; blue, 0 }  ,fill opacity=1 ] (129.2,217.14) .. controls (129.2,219) and (130.66,220.5) .. (132.47,220.5) .. controls (134.28,220.5) and (135.74,219) .. (135.74,217.14) .. controls (135.74,215.28) and (134.28,213.78) .. (132.47,213.78) .. controls (130.66,213.78) and (129.2,215.28) .. (129.2,217.14) -- cycle ;
\draw  [fill={rgb, 255:red, 255; green, 255; blue, 255 }  ,fill opacity=1 ] (129.2,232.84) .. controls (129.2,234.69) and (130.66,236.2) .. (132.47,236.2) .. controls (134.28,236.2) and (135.74,234.69) .. (135.74,232.84) .. controls (135.74,230.98) and (134.28,229.47) .. (132.47,229.47) .. controls (130.66,229.47) and (129.2,230.98) .. (129.2,232.84) -- cycle ;

\draw (155.13,125.66) node [anchor=north west][inner sep=0.75pt]  [font=\scriptsize]  {${\textstyle \frac{2\pi }{k}}$};
\draw (300.45,129) node [anchor=north west][inner sep=0.75pt]  [font=\scriptsize]  {$y^*_{i}$};
\draw (260.5,151.65) node [anchor=north west][inner sep=0.75pt]  [font=\scriptsize] [align=left] {$\displaystyle m$ points at distance $\displaystyle \ell $};
\draw (373.09,54.89) node [anchor=north west][inner sep=0.75pt]  [font=\scriptsize]  {$z^*_{2h}$};
\draw (379.76,74.84) node [anchor=north west][inner sep=0.75pt]  [font=\scriptsize]  {$z^*_{2h-1}$};
\draw (402.25,94.87) node [anchor=north west][inner sep=0.75pt]  [font=\scriptsize] [align=left] {$\displaystyle 2n$ points at distance $\displaystyle \bar \ell $};
\draw (214.62,114.06) node [anchor=north west][inner sep=0.75pt]  [font=\scriptsize]  {$\bar \ell'$};
\draw (114.61,151.16) node [anchor=north west][inner sep=0.75pt]  [font=\scriptsize]  {$0$};
\draw (138.33,212.05) node [anchor=north west][inner sep=0.75pt]  [font=\scriptsize] [align=left] {= peak of $\displaystyle U^*_{1}$};
\draw (138.33,226.8) node [anchor=north west][inner sep=0.75pt]  [font=\scriptsize] [align=left] {= peak of $\displaystyle U^*_{2}$};
\draw (398.76,120.84) node [anchor=north west][inner sep=0.75pt]  [font=\scriptsize]  {$z^*_{1}$};
\draw (402.45,141) node [anchor=north west][inner sep=0.75pt]  [font=\scriptsize]  {$y^*_{m+1}=z^*_0$};
\draw (239.12,129.67) node [anchor=north west][inner sep=0.75pt]  [font=\scriptsize]  {$y^*_{1}$};
\draw (223.56,87.9) node [anchor=north west][inner sep=0.75pt]  [font=\scriptsize]  {$R_{k} y^*_{1}$};

\end{tikzpicture}  \caption{The configuration set of peaks} 
\end{figure}

\medskip \noindent  The configuration is meaningful provided the {\em geometric  condition} $2 m\ell \sin \frac{\pi}{k}+\bar \ell'=2n \bar \ell$ does hold. The parameters $\bar \ell'$and $\bar \ell$ are related to $\ell$ via a balancing condition at $y_1^*$ and $y_{m+1}^*$ (see \eqref{yj*}-\eqref{zh*}), which is the same at both points in \cite{musso-pacard-wei} and allows to determine  $\bar \ell'=\bar \ell$ in terms of $\ell$. In this way, infinitely many solutions $W_j$ of \eqref{lim-nod-2} are found in \cite{musso-pacard-wei} with $\ell_j \to +\infty$ and diverging sequences $m_j,\ n_j \in \mathbb{N}$.\\

In order to state our main result, let us introduce the assumptions required on the various exponents in \eqref{ests}.  First of all,  we  assume 
\begin{equation} \label{po} \hbox{$1<{\mathfrak p}<+\infty$ if  $N=2$\quad and\quad $1< {\mathfrak p}<\frac{N+2}{N-2}$ if $N\geq 3$.}
\end{equation}
Introducing the following relevant quantities
\begin{equation} \label{557}
a:=\min\{a_1+1,a_2\}, \quad b:=\min\{b_2+1,b_1\},  \quad c=\min\{a_1,a_2,b_1,b_2\},
\end{equation}
let us   assume 
\begin{equation} \label{ab}
c>1, \  \max\{a, b\}<c+1.
\end{equation}
Necessary condition to have \eqref{ab} is
\begin{equation}\label{0849}
a_2<a_1+1,\ b_1<b_2+1
\end{equation}
in view of $a<c+1\leq a_1+1$ and $b<c+1\leq b_2+1$.  Since the peaks configuration in Figure \ref{figure} is not symmetric in the two components,  we can use both $(U_1^*,U_2^*)$ and $(U_2^*,U_1^*)$ as an approximate solution to \eqref{ests}. Up to exchange $u_1$ and $u_2$ in \eqref{ests}, hereafter we assume $a\leq b$ and we consider the approximate solution $U^*$ as $U^*=(U_1^*,U_2^*)$. Our main result reads as follows.
\begin{thm}\label{main}
Assume \eqref{po} with ${\mathfrak p}\geq 2$  and  \eqref{ab}. Given $(m,n,k)$ be an admissible triplet (see Definition \ref{admi}), there exists $\Lambda_0>0$ such for that for any $\Lambda>\Lambda_0$ the system \eqref{ests} has a non-radial positive solution $u$ branching off from $U^*$.
\end{thm}
Since there exist infinitely many admissible triplets $(m,n,k)$ (as stated in Lemma \ref{admi2}), Theorem~\ref{main} yields infinitely many non-radial positive solutions of \eqref{ests}. To the best of our knowledge, this provides the first example of non-synchronized solutions for strongly competitive Schr\"odinger systems in the whole space. This result should be contrasted with the cooperative case $\Lambda<0$, where Busca and Sirakov \cite{bus-sir} proved that all positive solutions are radially symmetric.\\

Compared with the construction in  \cite{musso-pacard-wei} , where the number of peaks diverges and plays the role of a parameter in the Lyapunov--Schmidt procedure, the presence of the coupling parameter $\Lambda$ allows us to fix any admissible triplet $(m,n,k)$. On the other hand, the strong competition regime introduces significant technical difficulties. In particular, while the interaction among peaks of the same component is governed by a superposition principle, this property fails for the interaction between different components unless $a_2=b_1=1$. As a consequence, delicate expansions are required to control the corresponding interaction forces. Remarkably, despite this lack of linear superposition, both the size and the direction of the leading-order interaction terms turn out to be consistent with those obtained by formal superposition.


Several remarks concerning the assumptions of Theorem~\ref{main} are in order.\\

\begin{rem}  
\rm 
Our result does not apply   to the the Bose-Einstein system \eqref{bose}.
Indeed, the choice $b_1=a_1+1$ and $b_2=a_2-1$, which corresponds to a variational structure for system \eqref{ests}, does not satisfy the assumption \eqref{ab}.  In particular, we could say that such a choice is {\em critical} for the inequality \eqref{ab} since $a=b=c+1$. This means that the approximate solution $U^*$ is not good enough to perform a perturbation argument and one should identify the small correction term to add into $U^*$ in order to improve the smallness of the corresponding error. 
  \\
In this regard, it is useful to note that the {\em sub-critical} choice 
$b_1=a_1+1-\epsilon$ and $b_2=a_2-1+\epsilon$ with $a_1>1$ meets the requirement \eqref{ab} provided 
$a_2 < a_1+1<a_2+2\epsilon$ and $\epsilon>0$ sufficiently small (in the form $0<\epsilon<\min\{1,\frac{a_1}{2},a_1-1 \}$).\\
We wish   to point out that the existence of non-radial solutions for the Bose-Einstein system \eqref{bose} is still open in a strongly competitive regime. At this aim it is worth to point out that the existence of non-radial solutions to \eqref{bose} (when the linear terms in   both equations are replaced by $\lambda_i u_i$ and $\lambda_1\not=\lambda_2$) has been obtained by Chen, Lin and Wei in \cite{chen-lin-wei}  in a weakly competitive regime (i.e. $\Lambda>0$ and small). \end{rem}
\begin{rem} \rm  
   We also observe that our result does not apply to systems with Lotka--Volterra type interactions, namely when $a_1 = b_1 = q = 1$ and $a_2 = b_2 = r = 1$. Indeed, the parameters $q$ and $r$ satisfy~\eqref{ab} only if $r - 1 < q < r + 1$, and therefore the case $q = r = 1$ is not covered. 
It would be interesting to investigate whether the more general configurations introduced by Ao, Musso, Pacard, and Wei in~\cite{ao-mu-pa-we} could be used to address this case and, more generally, systems that do not satisfy~\eqref{ab}.\\
In this regard, we would like to mention that very recently Clapp and Szulkin~\cite{clapp-szulkin} found solutions to a class of non-variational systems—including the Lotka--Volterra and Bose--Einstein interaction types—in the weakly competitive regime with $\Lambda > 0$ fixed.

\end{rem}
\begin{rem} 
\rm  
Assumption $\mathfrak p \geq 2$ is not really needed, but it simplifies the statement of Theorem~\ref{main}. If we allow $\frac{3}{2} < \mathfrak p < 2$, then in~\eqref{d5} the additional conditions~\eqref{0945} and~\eqref{1642}--\eqref{1653} have to be imposed, where $\mu$ is defined in~\eqref{drate}. Accordingly, assumption~\eqref{ab} has to be modified in order to ensure that the range in~\eqref{d5} is non-empty. The case $1 < \mathfrak p \leq \frac{3}{2}$ cannot be treated.\\
We also point out that the lower bound $\mathfrak p \geq 2$, together with the sub-critical decay, holds only in low dimensions, i.e.\ for $2\le N \le 5$.

\end{rem}

\begin{rem} \rm
Finally, it would be interesting to exploit the possibility to build this kind of solutions for system with $d$ components
as
$$-\Delta u_i+u_i=u_i^\mathfrak p-\sum\limits_{j=1\atop i\not=j}^d\Lambda_{ij}u_i^{a_{ij}}u_j^{b_{ij}}\ \hbox{in}\ \mathbb R^N,\ i=1,\dots,d,$$
in a mixed cooperative and competitive regime (i.e.  some $\Lambda_{ij}$'s are positive and  some are negative).
\end{rem}

The paper is organized as follows. In Section \ref{theansatz} we introduce the ansatz and the balancing conditions to explain the role of \eqref{ab} in the correct choice of $\bar \lambda'$, $\lambda$ and $\bar \lambda$. In Section \ref{liapu-pro} we perform the Ljapunov-Schmidt reduction and the expansions for the reduced problem are carried out in Section \ref{sec1122}.  Finally,  the proof of Theorem \ref{main} is given in Section \ref{secmain}.\\


 \section{The ansatz} \label{theansatz}
Denote by  $\mathtt e_i$, $i=1,\dots,N$, the canonical basis of $\mathbb{R}^N$. Let $R_k$ be the rotation of angle $\frac{2\pi}k$ in the $(x_1,x_2)-$plane and $R_{-k}:=R_k^{-1}$.  In Figure \ref{figure} the {\em inner} regular polygon in $\mathbb R^2\times\{0\} \subset \mathbb{R}^N$ has vertices given by the orbit of 
\begin{equation}\label{y1star}
y^*_1:=\frac{\bar\ell'}{2\sin{\pi\over k} }\mathtt{e}_1 
\end{equation}
under the group action generated by $R_k$.  By construction, the edges of this polygon have length $\bar\ell'$. Next, the {\em outer} regular polygon in Figure \ref{figure} has vertices generated by the orbit of 
$$z^*_0=y^*_{m+1} := y^*_1 + m \ell \mathtt e_1$$
under the group action generated by $R_k$.  Since the distance of $y^*_{m+1} $ from the origin is $m\ell+ \frac{\bar\ell'}{2\sin(\pi/k)}$,  to have the edges of the outer polygon of length $2n\bar\ell$ the involved parameters must satisfy,  as already said, the following {\em geometric  condition}
\begin{equation} \label{0946}
\fbox{$2 m\ell \sin \frac{\pi}{k}+\bar\ell'=2n\bar\ell$}
\end{equation} 
Let us consider the evenly distributed points
\begin{equation}\label{yj*}
y^*_j :=y^*_1+(j-1) \ell\mathtt e_1, \:  j=1,\dots, m+1,
\end{equation}
along the ray joining $y^*_1$ and $y^*_{m+1}$ and
\begin{equation}\label{zh*}
z^*_h :=y^*_{m+1} +h\bar\ell \mathtt t, \: h=0,\dots ,2n,
\end{equation}
along the edge of the outer polygon of length $2n\bar\ell$,  where
\begin{equation}\label{vect}
\mathtt t := - \sin \frac \pi k \mathtt e_1 + \cos \frac\pi k \mathtt e_2 ,  \quad \mathtt n :=  \cos \frac \pi k \mathtt e_1 + \sin \frac\pi k \mathtt e_2 .
\end{equation}
By construction notice that $z_{2n}^*=R_k y^*_{ m + 1}.$ Finally, we define $(U_1^*,U_2^*)$  as
\begin{equation}\label{uistar}
 \begin{aligned}U^*_1(x):=&\sum_{i=0}^{k-1} \left[ \sum_{j=1}^m U(x - R_k^{i}y_j^*) + \sum_{h=0}^{n-1}U(x - R_k^{i} z_{2h}^*)\right]\\
U^*_2(x):=&\sum_{i=0}^{k-1}   \sum_{h=1}^{n}U(x - R_k^{i} z_{2h-1}^* ),
 \end{aligned} 
\end{equation}
where $R_k^i$ denotes the $i$th power of $R_k$.  The function $(U_1^*,U_2^*)$ is an approximate solution of \eqref{ests} since the distance among all peaks will tend to infinity as  $\Lambda \to +\infty$.\\

Due to the symmetries of $(U_1^*,U_2^*)$ it is natural to look for solutions $(u_1,u_2)$ of \eqref{ests} which are invariant 
as follows:
\begin{equation}\label{roto}
u(x)=u(R_kx),\qquad u(\dots,x_i,\dots)=u(\dots,-x_i,\dots), \: i=2,\dots,N.
\end{equation}
Since the linearized operator of \eqref{ests} at $(U_1^*,U_2^*)$ has a kernel due to translation almost-invariances,  we need to compensate it by introducing some translation in the peaks.  The approximate solution $U=(U_1,U_2)$
we are interested in is defined as
\begin{equation} 
\label{u1}
U_1(x):=\sum_{i=0}^{k-1} \left[ \sum_{j=1}^m U(x- R_k^{i}y_j) + \sum_{h=0}^{n-1}U(x - R_k^{i} z_{2h})\right]=\sum\limits_{p\in\Pi_1}
U(x-p)
\end{equation}
and
 \begin{equation} 
\label{u2}U_2(x):=  \sum_{i=0}^{k-1} \sum_{h=1}^{n}U(x - R_k^{i} z_{2h-1}) = \sum\limits_{p\in\Pi_2}
U(x-p),
 \end{equation}
where
\begin{equation}\label{yj}
y_j=y_j^*+\al_j \mathtt e_1, \:  j=1,\dots,m ,
\end{equation}
and
\begin{equation}\label{zh}
z_0=z_0^*+\alpha_{m+1} \mathtt e_1,\quad z_h=z_h^*+\be_h \mathtt t + \bar\ell \ga_h \mathtt n,\:  h=1,\dots,2n-1.
\end{equation}
The notation $\Pi=\Pi_1\cup \Pi_2$ denotes the set of all  peaks which collects those of the first component $U_1$  
$$\Pi_1:=\bigcup_{i=0}^{k-1}\Big(\{R_k^{i} y_j : j=1,\dots,m\}\cup \{R_k^{i}z_{2h} : h=0,\dots,n-1\}\Big)$$
and the ones of the second component $U_2$  
$$\Pi_2:=\bigcup_{i=0}^{k-1} \{R_k^{i}z_{2h-1} : h=1,\dots,n\}.$$
When $\alpha_j=\beta_h=\gamma_h=0$ for all $j=1,\dots,m+1$ and $h=1,\dots,2n-1$, the sets $\Pi,\Pi_1,\Pi_2$ will be simply denoted by $\Pi^*,\Pi_1^*,\Pi_2^*$.  On the translation parameters let us assume 
 \begin{equation}\label{cpabg}
|\al_j|\le  1 \quad \text{ for } j=1,\dots,m+1, \quad |\be_h|+\bar\ell |\ga_h| \le   1 \quad \text{ for } h=1,\dots,2n-1,
\end{equation}
and
\begin{equation} \label{symme}
 \beta_h=-\beta_{2n-h}, \quad \gamma_h=\gamma_{2n-h} \qquad \hbox{for} \ h=1,\dots,n,
\end{equation}
where \eqref{symme} simply guarantee the invariance \eqref{roto} for $U$.  The free parameters $\alpha_j$'s,  $\beta_h$'s and $\gamma_h$'s form a $(2n+m)-$dimensional family in view of \eqref{symme} and a-posteriori will be very small, yielding that $U$ is a small perturbation of the initial ansatz $(U_1^*,U_2^*)$.

\subsection{The  balancing conditions}
The function $(U^*_1,U^*_2)$ depends on three integer parameters $n,$ $m$ and $ k$ and three continuous ones $\ell,$ $\bar \ell$ and $\bar\ell'$.  We will fix a triple $(n,m,k)$,  admissible in a suitable sense, and will take the continuous  parameters large as $\Lambda\to+\infty$ to have $(U^*_1,U^*_2)$ as an approximate solution of \eqref{ests}.  In addition to the {\em geometric condition} \eqref{0946}
they must also satisfy two {\em  balancing conditions} which can be deduced as follows.\\

Let us introduce the interaction functions which measure how each peak  interacts with the ones in the same or different component.  Given $p \in \Pi_i$,  denote by $\Pi_j^p$ the set of the closest peaks in $\Pi_j$ to $p$. 
The  {\em inner} interaction is concerned with peaks of the same component:
\begin{equation}\label{in-1}
{\boldsymbol\Psi}^{(0)}(p):=\frac1{\mathfrak p}\int_{\mathbb{R}^N} \left(\sum\limits_{q\in \Pi_i^p}U(x+p-q)\right)\nabla U^{{\mathfrak p}}(x)dx,\: p\in \Pi_i.
\end{equation}
According to \cite{musso-pacard-wei} the contribution coming from each $q \in \Pi_i$ has the form 
\begin{equation} \label{1531}
\frac1{\mathfrak p} \int_{\mathbb{R}^N} U(x+p-q) \nabla U^{{\mathfrak p}}(x)dx=\Psi^{(0)} (|p-q|) \frac{p-q}{|p-q|},
\end{equation}
where
\begin{equation}\label{psizero}
\Psi^{(0)} (t)=t^{-\frac{N-1}{2} } e^{-  t }({\mathfrak c}_0+o(1))\ \hbox{as}\ t\to\infty
\end{equation}
for some ${\mathfrak c}_0>0$, see \eqref{1328}.  Expansion \eqref{psizero} follows by the asymptotics of the radially symmetric solution $U$ of \eqref{bubble}:
\begin{equation} \label{1819}
\lim_{r \to +\infty} e^r r^{\frac{N-1}{2}}  U(r)>0, \quad \lim_{r \to +\infty} \frac{U'(r)}{U(r)}=-1.
\end{equation}
The asymptotic size $e^{-|p-q|}$ of the contribution coming from $q \in \Pi_i$ explains why at main order the sum
in \eqref{in-1} is just over $q \in \Pi_i^p$. \\

Up to $\Lambda$, the {\em outer} interactions,  concerning peaks of different components, take the form:
\begin{eqnarray}\label{in-2_1}
&&{\boldsymbol\Psi}^{(1)}(p):=\frac1{a_1+1}\int_{\mathbb{R}^N} \Big[\sum_{q \in \Pi_2^p} U(x+p-q)\Big]^{a_2} \nabla U^{a_1+1}(x)dx,\:
p\in \Pi_1, \\
\label{in-2_2}
&& {\boldsymbol\Psi}^{(2)}(p):=\frac1{b_2+1}\int_{\mathbb{R}^N} \Big[\sum_{q \in \Pi_1^p} U(x+p-q) \Big]^{b_1} \nabla U^{b_2+1}(x)dx,\: p\in \Pi_2.
\end{eqnarray}
As already noticed in the Introduction,  the {\em outer} interactions don't decompose as superposition of the contributions generated by each other peak since $a_2,b_1>1$ and delicate expansions will be carried out in Section \ref{new-1}.  For the discussion here,  let us just recall \eqref{1604}-\eqref{psi1barell-2}: ${\boldsymbol\Psi}^{(1)}(y_{m+1}^*)=-\Psi^{(1)}(\bar\ell) \mathtt e_1$ with
\begin{equation}\label{psiuno}\Psi^{(1)}(\bar\ell)=\bar \ell^{-\frac{N-1}{2}a_2} e^{-a_2 \bar \ell }(\mathfrak c_1+o(1))\ \hbox{as}\ \bar \ell \to\infty \end{equation}
for some $\mathfrak c_1>0$.  The presence of $\Pi_j^p$ in \eqref{in-2_1}-\eqref{in-2_2} can be justified by the estimates
\begin{eqnarray}
\int_{\mathbb{R}^N}U^{a_2}(x+p-q) \nabla U^{a_1+1}(x)dx=\mathcal O\Big(|p-q|^{-a_2 \frac{N-1}{2}} e^{-a_2 |p-q|}\Big)\label{1532} 
\end{eqnarray}
for $p \in \Pi_1$,  $q \in \Pi_2$,  and
\begin{eqnarray}
\int_{\mathbb{R}^N} U^{b_1}(x+p-q) \nabla U^{b_2+1}(x)dx=\mathcal O\Big(|p-q|^{-b_1 \frac{N-1}{2}} e^{-b_1 |p-q|}\Big),\label{1533}
\end{eqnarray}
for $p \in \Pi_2$,  $q \in \Pi_1$, which are valid in view of \eqref{ab}-\eqref{0849}.\\

The    {\em inner balancing condition} coincides with the one found in \cite{mu-pa-wei-2012} and states the equilibrium among the forces acting on $y_1^*$ (see Lemma \ref{new-2}):
\begin{equation}\label{ba1}
\fbox{$\Psi^{(0)}(\ell)-2\Psi^{(0)}(\bar\ell')\sin\frac\pi k=0$}
\end{equation}
in view of $\Pi_1^{y_1^*}=\{y_2^*,R_k y_1^*,R_{-k} y_1^*\}$ and 
$${\boldsymbol\Psi}^{(0)}(y_1^*)=\Big(\Psi^{(0)}(\ell)-2\Psi^{(0)}(\bar\ell')\sin\frac\pi k\Big)\mathtt e_1.$$
The {\em outer balancing condition} is new and arises from the equilibrium among the forces acting on $y_{m+1}^*$ (see Lemmas \ref{new-1}-\ref{new-2}):
\begin{equation}\label{ba2}
\fbox{$\Psi^{(0)}(\ell)-\Lambda  \Psi^{(1)}(\bar\ell)=0$}
\end{equation}
in view of $\Pi_1^{y_{m+1}^*}=\{y_m^*\}$, $\Pi_2^{y_{m+1}^*}=\{z_1^*,R_{-k} z_{2n-1}^*\}$ and
$${\boldsymbol\Psi}^{(0)}(y_{m+1}^*)-\Lambda {\boldsymbol\Psi}^{(1)}(y_{m+1}^*)=\Big( \Psi^{(0)}(\ell)-   \Lambda \Psi^{(1)}(\bar\ell)\Big)\mathtt e_1.$$
Notice that the forces acting on all the other peaks are automatically in equilibrium. In particular, since ${\boldsymbol\Psi}^{(2)}(z_{2h-1}^*)=0$ for $h=1,\dots, n$, an expression for the modulus of ${\boldsymbol\Psi}^{(2)}$ as in \eqref{psiuno} is not available and we define $\Psi^{(2)}(\bar\ell)$ in a artificial way as follows
\begin{equation}\label{psidue}\Psi^{(2)}(\bar\ell)=\bar \ell^{-\frac{N-1}{2}b_1} e^{-b_1 \bar \ell }.
\end{equation}
Such a definition is still meaningful since, as we will see in Lemma \ref{new-1}, when we will perturb the point $z_{2h-1}^*$ into $z_{2h-1}$, the size of ${\boldsymbol\Psi}^{(2)}(z_{2h-1})$ will be precisely given by \eqref{psidue}.\\

The balancing conditions \eqref{ba1}-\eqref{ba2}, together with the geometric condition \eqref{0946}, allow to choose the parameters $\ell,\bar\ell$ and $\bar\ell'$ in terms of the free parameter $\Lambda$ as follows
\begin{equation}\label{expansion}
\log \Lambda=[a\mu-1] \ell +\mathcal  O(\log \ell), \quad \bar\ell =\mu \ell +\mathcal  O(1) ,\quad \bar\ell' =\ell +\mathcal O(1) ,
\end{equation}
where  
\begin{equation}\label{drate} \mu=\mu(m,n,k):= \frac{2m\sin \frac{\pi}{k}+1}{2n}\end{equation} 
has to satisfy $\mu >\frac{1}{a}$ to have \eqref{expansion} meaningful. Indeed, by \eqref{psizero} and \eqref{ba1} we deduce that  $ \bar\ell'=\ell+\mathcal O(1)$, which inserted in \eqref{0946}, gives that
\begin{equation}0=2 m\ell \sin \frac{\pi}{k}-2n\bar\ell +\bar\ell'=(2 m\sin \frac{\pi}{k}+1)\ell -2n\bar\ell+\mathcal O(1).
\label{1229}\end{equation}
Moreover, by  \eqref{psizero}, \eqref{psiuno} and \eqref{ba2}  
\begin{equation}-\ell =\log \Lambda-a\bar\ell-\frac{N-1}{2}(a-1) \log \bar \ell +\mathcal  O(1),\label{1230}\end{equation}
and then \eqref{expansion} follows by \eqref{1229}-\eqref{1230}. \\

To identify the strongest interaction forces acting on $p$,  we need to compare the size of ${\boldsymbol\Psi}^{(0)}(p)$ with $\Lambda {\boldsymbol\Psi}^{(1)}(p)$,  when $p \in \Pi_1$,  or $\Lambda {\boldsymbol\Psi}^{(2)}(p)$, when $p \in \Pi_2$, and we will denote $\Pi^p \subset \Pi_1^p \cup \Pi_2^p$ as the set of peaks whose interaction force on $p$ is predominant. Thanks to  \eqref{1531}-\eqref{psizero} and \eqref{1532}-\eqref{1533} we can characterize $\Pi^p$ as the set of ``closest" peaks with respect to $d(p,\cdot)$, where
$$
d(p,q)=\left\{ \begin{array}{ll}
|p-q|  &\hbox{if }p,q \in \Pi_1 \hbox{ or }p,q \in \Pi_2 \\
a|p-q|-\log \Lambda & \hbox{if }p\in \Pi_1, \ q \in \Pi_2\\
b|p-q|-\log \Lambda  & \hbox{if }p\in \Pi_2,\ q \in \Pi_1.
\end{array} \right.
$$
For example, the two balancing condition \eqref{ba1}-\eqref{ba2} hold in view of 
\begin{equation} \label{1623}
\Pi^{y_1^*}=\{y_2^*,R_k y_1^*,R_{-k} y_1^*\}, \quad \Pi^{y_{m+1}^*}=\{y_m^*,z_1^*,R_{-k} z_{2n-1}^*\}.
\end{equation} 
 Besides them,  in the discussion below we also establish that 
\begin{equation} \label{1624}
\Pi^{y_j^*}=\{y_{j-1}^*,y_{j+1}^*\}, \ j=2,\dots,m, \quad \Pi^{z_h^*}=\{z_{h-1}^*,z_{h+1}^*\}, \ h=1,\dots,2n-1. 
\end{equation} 
Since
\begin{equation} \label{1838}
\frac{ |z_2^* -y_m^*|}{\ell} =\frac{|   \ell \mathtt e_1+ 2 \bar \ell \mathtt t|}{\ell} \to \sqrt{1+4\mu^2-4\mu \sin\frac\pi k}>1
\end{equation}
when $\mu >\sin\frac\pi k$, observe that
\begin{eqnarray}
\rho_1^*:&=&\min\limits_{p,q\in \Pi_1^* \atop p\ne q}|p-q| =\min \{ |R_k y_1^*-y_1^*|, |y^*_{i+1}-y_i^*|,|z^*_{2h+2}-z^*_{2h}|, |z^*_2-y^*_m|\} \nonumber \\
&&=\min \{ \bar \ell', \ell,2\bar \ell, |z_2^*-y_m^*|\}= \ell \label{1117}
\end{eqnarray}
in view of \eqref{expansion} if $\mu>\frac{1}{2}$ and $k\geq 6$.  Since $|y^*_m-y^*_{m+1}| = \ell$, in particular $y^*_m$ is the closest point of $\Pi_1$ to $y^*_{m+1}$ with an interaction of size $\Psi^{(0)}(\ell)$.  On the other side, we have that
\begin{eqnarray} 
\rho^*: = \min\limits_{p\in \Pi_1^*, q\in \Pi_2^*}|p-q|
=\min\{|z_1^*-y^*_{m+1}| , |z^*_1-y^*_m| \} = \min\{\bar \ell , |z_1^*-y_m^*| \}=\bar \ell  \label{1118}
\end{eqnarray}
if $\mu < 1$ and $k \geq 6$ in view of  \eqref{expansion} and 
\begin{equation}
\frac{ |z_1^*-y_m^*|}{\bar \ell}=\frac{|   \ell \mathtt e_1+  \bar \ell \mathtt t|}{ \mu \ell+\mathcal  O(1)} \to \sqrt{1+\mu^{-2}-2 \mu^{-1} \sin\frac\pi k} >1 \label{2001}
\end{equation}
when $\mu^{-1} >2 \sin\frac\pi k$. Since $|z^*_1-y^*_{m+1}|, |R_{-k}z^*_{2n-1}-y^*_{m+1}|=\bar \ell$,  in particular $z^*_1,R_{-k}z^*_{2n-1}$  are the closest points of $\Pi_2$ to $y^*_{m+1}$ with an interaction of size $\Lambda \Psi^{(1)}(\bar\ell)$ and then $\Pi^{y_{m+1}^*}=\{y_m^*,z_1^*,R_{-k} z_{2n-1}^*\}$ in view of \eqref{ba2}.  
It is easily seen that $\Pi^{y_1^*}=\{y_2^*,R_k y_1^*,R_{-k} y_1^*\}$ in view of \eqref{expansion}. Similarly, 
$\Pi^{y_j^*}=\{y_{j-1}^*,y_{j+1}^*\}$ for $j=2,\dots,m$ and the predominant forces are automatically in equilibrium for symmetry reasons, as we will see. For $y^*_m$ this follows by
$$\frac{a |z^*_1-y^*_m| -\log \Lambda}{|y^*_{m+1}-y^*_m|} \to a \sqrt{1+\mu^2-2\mu \sin\frac\pi k} - (a \mu-1)>1$$
in view of \eqref{expansion}  and \eqref{2001}. Finally,  since
\begin{eqnarray}\label{1119}
\rho^*_2:&&=\min\limits_{p,q\in \Pi_2^* \atop p\ne q}|p-q|=\min\{|z^*_{2h+1}-z^*_{2h-1}|,  |z^*_1-R_{-k}z^*_{2n-1}|\} \nonumber \\ 
&& =\min\{2\bar \ell,  |z_1^*-R_{-k}z_{2n-1}^*|\}
=2\cos  \frac{\pi}{k}\bar \ell,
\end{eqnarray}
we have that $\Pi^{z_h^*}=\{z_{h-1}^*,z_{h+1}^*\}$ for $h=1,\dots,2n-1$ with prevalent forces still in equilibrium for symmetry reasons, as it follows by
$$\frac{|y^*_m-z^*_2|}{a|z^*_1-z^*_2|-\log \Lambda}=\frac{|y_m^*-z_2^*|}{\ell+\mathcal  O(1)}\to \sqrt{1+4\mu^2-4\mu \sin\frac\pi k}>1$$
and
$$\frac{|z^*_1-R_{-k} z^*_{2n-1}|}{b |y^*_{m+1}-z^*_1|-\log \Lambda} \to  \frac{2 \mu \cos  \frac{\pi}{k}}{(b-a)\mu+1} >1$$
in view of  \eqref{expansion},  \eqref{1838} and \eqref{fund} provided
$$\mu >\frac 1{a-b+2\cos\frac\pi k}.$$
Notice that $a-b+2\cos\frac\pi k>a-c+2\cos\frac\pi k-1\geq 2\cos\frac\pi k-1\geq 0$ if $k\geq 6$ in view of  \eqref{ab}.  Moreover,  thanks to \eqref{cpabg} estimates \eqref{1117}-\eqref{1118} and \eqref{1119} remain valid in the form
\begin{equation} \label{1217}
\begin{aligned}
\rho_1:=& \min\limits_{p,q\in \Pi_1 \atop p\ne q}|p-q| = \ell+\mathcal O(1), \quad \rho: = \min\limits_{p\in \Pi_1, q\in \Pi_2}|p-q|=\mu \ell+\mathcal O(1)\\ 
\rho_2:=& \min\limits_{p,q\in \Pi_2 \atop p\ne q}|p-q|=2 \mu \cos  \frac{\pi}{k} \ell+\mathcal O(1)
\end{aligned}
\end{equation}
in view of \eqref{expansion}.\\

As a result of such a discussion, hereafter assume $k\geq k_0$, where $k_0\geq 6$ is sufficiently large to have
\begin{equation} \label{fund}
b<c-1+2\cos \frac{\pi}{k_0}
\end{equation}
in view of \eqref{ab},  and $\mu$ to satisfy
\begin{equation}\label{d5}
\max\left\{\frac1a,\frac{1}{2},\frac 1{a-b+2\cos\frac\pi k}\right\}<\mu<\frac{1}{a+1-c}.
\end{equation}
The role of $a+1-c$ will become clear in the sequel and condition $(a+1-c)\mu<1$ is essential to ensure that the projection of \eqref{pro-pro} onto the linearized kernel is at main order due to the error term $\mathscr E$ (see  \eqref{0845} and Lemma \ref{linear-rate}).  Condition $\mu>\frac{1}{a}$ is necessary in \eqref{expansion} and is consistent with $\mu<\frac1{a+1-c}$ since $1<c \leq a$ in view of \eqref{ab}. Moreover, observe that $a+1-c<2$ and $a+1-c<a-b+2\cos \frac{\pi}{k}$ follow by \eqref{ab} and \eqref{fund},  to guarantee that assumption \eqref{d5} makes perfectly sense.\\

Finally,  we are ready to introduce the following definition.
\begin{defi}\label{admi}We say that $(m,n,k)$ is admissible if
\eqref{d5} holds true.
\end{defi}
In order to explore the range of applicability of Theorem \ref{main}, we prove the existence of  infinitely many   admissible triplets.
\begin{lm}\label{admi2}
For any $k\ge k_0$ there exists a sequence $(m_i,n_i)_{i\in\mathbb N}$ sucht that $(m_i,n_i,k)$ is admissible.
\end{lm}
\begin{proof}
Fix $k\ge k_0$ so that
$$\mathtt x:=\max\left\{\frac1a,\frac{1}{2},\frac 1{a-b+2\cos\frac\pi k}\right\}<\frac{1}{a+1-c}=:\mathtt y.$$
Choose  $q\in\mathbb Q$ such that $\mathtt x\le q\sin\frac\pi k<\mathtt y$ and let $m_i,n_i\in \mathbb N$ with $n_i\to\infty$ such that $q=\frac{m_i}{n_i}$. Since $ \mu(m,n,k)=\frac mn\sin\frac\pi k+\frac1{2n}$, then 
$$\mathtt x<\mathtt x+\frac1{2n_i} \le \mu(m_i,n_i,k)=q\sin\frac\pi k+\frac1{2n_i}<\mathtt y$$
for $i$ large enough and the claim is proved.
\end{proof}

 \section{The Ljapunov-Schmidt procedure} \label{liapu-pro}
To get positive solutions to \eqref{ests},   observe that we can find solutions in $H^1(\mathbb{R}^N)$ to 
\begin{equation}
\label{ests+}
\left\{ \begin{array}{ll} 
- \lap u_1 + u_1 =f(u_1) -\Lambda A (u_1,u_2) &\text{in $\mathbb{R}^N$}\\
- \lap u_2 + u_2 =f(u_2)  -\Lambda B (u_1,u_2) &\text{in $\mathbb{R}^N$}\\
\end{array} \right.
\end{equation}
where $f(u)=u_+^{{\mathfrak p}}$, $u=u_+-u_-,$   $u_+=\max\{u,0\}$ and 
$$A(s,t):=|s|^{a_1-1} s |t|^{a_2},\quad B(s,t): =|s|^{b_1}|t|^{b_2-1}t.$$
Indeed,  multiplying the first/second  equation in \eqref{ests+} by $(u_1)_-/(u_2)_-\in H^1(\mathbb{R}^N)$, respectively,  we get immediately $u_1,\, u_2\ge0$ in $\mathbb{R}^N$ since $\Lambda>0.$\\

So we look for solutions to \eqref{ests+} as
$$u_1= U_1 +\phi_1\ \hbox{and}\ u_2= U_2+\phi_2,$$
where $U_1$  and $U_2$ are defined in \eqref{u1}-\eqref{u2} and the remainder term
$\phi=(\phi_1,\phi_2)$ is symmetric according to \eqref{roto} satisfying the $(2n+m)$ orthogonality conditions \eqref{orto1}-\eqref{orto2}.  Here the unknowns are $\phi_1,$ $\phi_2$ and the $(2n+m)-$parameters  $\alpha_j$,  $\beta_h$  and $\gamma_h$ as in \eqref{symme}.

\medskip \noindent As it is usual, the problem reduces to solve
\begin{equation}
\label{sys}
\mathscr L (\phi)+\mathscr E+\mathscr Q(\phi)=0,
\end{equation}
where  the linear operator $\mathscr L= (\mathscr L_1,\mathscr L_2)$ is defined as
\begin{equation}\label{L}
\begin{aligned}
\mathscr L_1(\phi):=& -\lap \phi_1+\phi_1- f'(U_1)\phi_1 +\Lambda \partial_s A(U_1,U_2)\phi_1+\Lambda \partial_t A(U_1,U_2)\phi_2\\
\mathscr L_2(\phi):=& -\lap \phi_2+\phi_2- f'(U_2)\phi_2 +\Lambda \partial_s B(U_1,U_2)\phi_1+\Lambda \partial_t B(U_1,U_2)\phi_2,
\end{aligned}
\end{equation} 
 the error term $\mathscr E=(\mathscr E_1,\mathscr E_2)$ is defined as
\begin{equation}\label{E}
\begin{aligned}
\mathscr E_1:=& -\lap   U_1+  U_1- f(U_1)+\Lambda A (U_1,U_2) \\
\mathscr E_2:=& -\lap   U_2+  U_2- f(U_2)+\Lambda B (U_1,U_2) 
\end{aligned}
\end{equation} 
 and  the higher order term $\mathscr Q=(\mathscr Q_1,\mathscr Q_2)$ is defined as
\begin{equation}\label{Q}
\begin{aligned}
& \mathscr Q_1(\phi):= -f(  U_1+\phi_1)+f(U_1)+f'(U_1)\phi_1\\
& +\Lambda\left(A(  U_1+\phi_1,  U_2+\phi_2)-A(U_1,U_2)-\partial_s A(U_1,U_2)\phi_1 -\partial_t A(U_1,U_2)\phi_2 \right)\\
& \mathscr Q_2(\phi):=-f(  U_2+\phi_2)+f(U_2)+f'(U_2)\phi_2  \\
& +\Lambda\left(B (  U_1+\phi_1,  U_2+\phi_2)-B(U_1,U_2)-\partial_s B(U_1,U_2)\phi_1-\partial_t B(U_1,U_2)\phi_2 \right).
\end{aligned}
\end{equation}
We will often make use of the following simple estimates
\begin{equation} \label{key}
e^{-\lambda_1 |x-p|-\lambda_2|x-q|} \leq e^{-\min\{\lambda_1,\lambda_2\}|p-q|},\quad U +|U'|=\mathcal O\Big(\frac{e^{-|x|}}{(1+|x|)^{\frac{N-1}{2}}} \Big)\hbox{ in }\mathbb{R}^N
\end{equation}
in view of \eqref{1819}.

\subsection{A linear problem}
Hereafter,  the point $p_j$, $j=1,\dots, m+2n$,  will stand for the $j-$th element of $y_1,\dots,y_{m},z_0,z_2,\dots, z_{2n-2},z_1,z_3,\dots,z_{2n-1}$. Denoting by $\cdot$ the inner product,  introduce the pair $\mathscr M(  {\bf h}):=(\mathscr M_1(  {\bf h}),\mathscr M_2 ({\bf h}))$ as follows 
\begin{equation}\label{MMM}
 \begin{aligned} &\mathscr M_1 (  {\bf h}):=\sum\limits_{i=0}^{k-1} \sum\limits_{j=1}^{m+n} (R^i_k h_j)\cdot \nabla U(x-R^i_k p_j)\\
&\mathscr M_ 2({\bf h}) :=\sum\limits_{i=0}^{k-1} 
\sum\limits_{j=m+n+1}^{m+2n} (R^i_k h_j)\cdot \nabla U(x-R^i_k p_j), \end{aligned}
\end{equation}
where ${\bf h}:=(h_1,\dots,h_{m+2n}) \in(\mathbb R\mathtt e_1)^{m+1}\times (\mathbb R\mathtt e_1+\mathbb R\mathtt e_2)^{2n-1}$.  Let $\mathscr L= (\mathscr L_1,\mathscr L_2)$ be the linear operator in \eqref{L}, where $\phi_1, \ \phi_2$  are symmetric as in \eqref{roto} and satisfy the orthogonality conditions:
\begin{equation}\label{orto1}
\int_{\mathbb{R}^N}\phi_1 \partial_1 U(x-p_j)dx=0\quad \forall \, j=1,\dots,m+1
\end{equation}
and 
\begin{equation}\label{orto2}
\int_{\mathbb{R}^N}\phi_2  \nabla U(x-p_j)dx=0 \quad \forall \,  j=m+2,\dots,m+2n. 
\end{equation}
Notice that \eqref{orto1}-\eqref{orto2} are equivalent to
\begin{equation}\label{orto}
\int_{\mathbb{R}^N}\phi_1 \  \nabla U(x-p)dx=\int_{\mathbb{R}^N}\phi_2 \ \nabla U(x-q)dx=0\quad  \forall \ p \in \Pi_1, \ q \in \Pi_2,
\end{equation}
thanks to the symmetries \eqref{roto} of $\phi_1$ and $\phi_2$.  Set
$$\mathcal P =\Big \{ \Pi \hbox{ of the form } \eqref{yj}-\eqref{symme} \hbox{ satisfying }\eqref{expansion} \Big\}, \quad \mathcal P_i=\{ \Pi_i:\, \Pi \in \mathcal P\}$$
and
$$X_{symm}=\Big\{g \in X \hbox{ satisfies }\eqref{roto}\Big\}.$$
Let $\|\cdot\|$ be a norm in $(\mathbb{R}^N)^{m+2n}$ and introduce the weighted norm
$$\|g\|_*=\|g_1\|_*+\|g_2\|_*,\quad  \|g_i\|_*=\sup_{x\in\mathbb{R}^N} \Big| (\sum_{p \in\Pi} e^{-\eta |x-p|} )^{-1} g_i(x)\Big|, \:  i=1,2,
$$
for $\eta \in (0,1)$.  Given the spaces $L^{\infty,*}_{symm}=\{(g_1,g_2) \in [L^\infty_{symm}(\mathbb{R}^N)]^2:\ \|g\|_*<\infty\}$ and
$$\mathcal H_1 =\{ \phi_1 \in H^1_{symm}(\mathbb{R}^N) \hbox{ satisfying } \eqref{orto1}\},\quad 
\mathcal H_2 =\{ \phi_2 \in H^1_{symm}(\mathbb{R}^N) \hbox{ satisfying } \eqref{orto2}\},
$$
the solvability theory for $\mathscr L$ is described as follows.
\begin{prop}\label{linear} Given $\eta \in (0,1)$, there exist $\Lambda_0,\ C>0$ so that, for any $\Lambda\geq \Lambda_0$, $\Pi \in \mathcal P$ and $g\in L^{\infty,*}_{symm}$,   there are unique solution $\phi \in \mathcal H:=\mathcal H_1 \times \mathcal H_2$ and ${\bf h}$ to
\begin{equation} \label{1520}
\mathscr L(\phi)=g+\mathscr M({\bf h})\ \hbox{in}\ \mathbb{R}^N
\end{equation}
and they satisfy
\begin{equation} \label{1835}
\|{\bf h} \|+\|\mathscr \phi \|_{*}  \leq C \|g\|_{*}.
\end{equation}
\end{prop}
\begin{proof} The proof follows the arguments in \cite[Section 3]{mu-pa-wei-2012}. Let $\phi  \in \mathcal H$ so that $\mathscr L(\phi) \in L^{\infty,*}_{symm}$. Multiplying the first/second component of $\mathscr L(\phi)-\mathscr M({\bf h})$ by $ \partial_k U(x-p_s)$ if $p_s \in \Pi_1/p_s \in \Pi_2$ and integrating by parts we deduce that
$$h_s+o(\|{\bf h}\|)=\mathcal O(\| \mathscr L(\phi)-\mathscr M({\bf h})  \|_*)+o(\|\phi\|_{*})$$
since 
$$\int_{\mathbb{R}^N} \mathscr L_i(\phi)\partial_k U(x-p_s)=\mathcal O(\Lambda e^{-\mu(a-1)\ell} \|\phi\|_*)=\mathcal O(e^{-(1-\mu+o(1))\ell} \|\phi\|_*)=o(\|\phi\|_{*})$$
in view of Lemma \ref{linear-rate}, \eqref{expansion}, $a\leq b$ and $\mu<1$ (recall \eqref{d5} and $a+1-c\geq 1$). Therefore we have that
\begin{equation} \label{2043}
\|{\bf h}\|=\mathcal O(\|\mathscr L(\phi)-\mathscr M({\bf h}) \|_*)+o(\|\phi\|_{*}).
\end{equation}
Setting $ L_1:=-\lap +1-f'(U_1)+\Lambda \partial_s A(U_1,U_2)$ and $L_2:= -\lap+1-f'(U_2)+\Lambda \partial_t B(U_1,U_2)$, the   function $\omega(x):=\sum\limits_{p\in \Pi} e^{-\eta|x-p|} $ satisfies:
$$L_i \omega \geq \Big[1- \eta^2-f'(U_i)\Big] \omega \quad \hbox{in }\mathbb{R}^N\setminus \Pi$$
in view of $\Lambda \partial_s A(U_1,U_2), \Lambda \partial_t B(U_1,U_2)\geq 0$.  Since $1- \eta^2-f'(U_i) \geq \frac{1- \eta^2}{2}$ in $A_R:=\mathbb{R}^N\setminus \displaystyle \bigcup_{p \in\Pi} B(p,R)$ for $R$ large, the operator $L_i$, $i=1,2$,  satisfies the weak maximum principle in $H^1(A_R)$. Since 
$$\|L_1 (\phi_1)\|_*=\|\mathscr L_1(\phi)-\Lambda \partial_t A(U_1,U_2) \phi_2\|_{*} \leq \epsilon_1=\|\mathscr L_1(\phi)-\mathscr M_1({\bf h})\|_{*}+C\|{\bf h}\| 
+\Lambda e^{-\mu(a-1)\ell}  \|\phi_2\|_{*}$$
and
$$\|L_2 (\phi_2)\|_*=\|\mathscr L_2(\phi)-\Lambda \partial_s B(U_1,U_2) \phi_1\|_{*} \leq \epsilon_2=\|\mathscr L_2(\phi)-\mathscr M_2({\bf h})\|_{*}+C\|{\bf h}\|  +\Lambda e^{-\mu(b-1)  \ell}  \|\phi_1\|_{*}$$
in view of \eqref{1217} and \eqref{key}, one can use $\omega \in H^1(A_R)$ as a barrier to establish the pointwise estimate
\begin{equation} \label{1250}
|\phi_i|\leq \Big(\frac{2\epsilon_i}{1-\eta^2}+e^{\eta R} \sup\limits_{p \in\Pi}\|\phi_i \|_{L^\infty(\partial B(p,R))}\Big) \omega \quad \hbox{ in }\mathbb{R}^N\setminus \bigcup_{p \in\Pi} B(p,R),
 \end{equation}
where $\phi \in C_{loc}(\mathbb{R}^N)$ follows by local elliptic regularity theory. Since $\|\phi\|_*<+\infty$ in view of \eqref{1250}, the following crucial estimate
\begin{equation} \label{1550}
\|{\bf h}\|+\|\phi \|_{*} =\mathcal O\Big(|\| \mathscr L(\phi)-\mathscr M({\bf h})\|_{*}\Big) \quad \forall \, \phi \in \mathcal H \hbox{ with }\mathscr L(\phi) \in L^{\infty,*}_{symm}
\end{equation}
makes perfectly sense and we aim to prove it. We can argue by contradiction and assume the existence of sequences $\Lambda \to +\infty$,  $\Pi \in \mathcal P$, ${\bf h}$ and $\phi \in \mathcal H\cap [W^{2,N}(\mathbb{R}^N)]^2$ so that
\begin{equation} \label{1522}
\|{\bf h}\|+\|\phi \|_{*}=1, \quad   \| \mathscr L(\phi)-\mathscr M({\bf h})\|_{*}  \to 0.
\end{equation}
Up to a subsequence, there exist $i_0=1,2$ and $p_0 \in \Pi$ so that
$$\|\phi_{i_0} \|_{L^\infty( B(p_0,R))}\ge \frac{e^{-\eta R}}{4}$$
in view of $\epsilon_1+\epsilon_2 \to 0$, \eqref{2043} and \eqref{1522}  thanks to $\Lambda e^{-(a-1) \mu \ell} \to 0$.  If $p_0 \in \Pi_{i_0}$, we apply Ascoli-Arzelà's theorem to the function $ \phi_{i_0} \left(\cdot+ p_0  \right)$ to deduce, up to a subsequence, the $L^\infty_{loc}-$convergence to a non-trivial bounded solution of
$$-\Delta \Phi+\Phi-{\mathfrak p} U^{{\mathfrak p}-1} \Phi=0\ \hbox{in}\ \mathbb{R}^N, \ \int_{\mathbb{R}^N} \Phi \nabla U dx=0,$$
in view of \eqref{expansion},  \eqref{1217}, \eqref{orto} and the $L^\infty_{loc}-$convergence
\begin{eqnarray*}
\Lambda[ |\partial_s A(U_1,U_2)| +|\partial_t A(U_1,U_2)|] (\cdot+p_0) =\mathcal O( \Lambda e^{- \mu (a_2-1)\ell})=\mathcal O(e^{-(1-\mu+o(1))\ell})  \to 0 \hbox{ if }i_0=1\\
\Lambda[ |\partial_s B(U_1,U_2)| +|\partial_t B(U_1,U_2)|] (\cdot+p_0) =\mathcal O( \Lambda e^{- \mu (b_1-1)\ell}) \leq  \Lambda e^{- \mu (a_2-1)\ell} \to 0 \hbox{ if }i_0=2
\end{eqnarray*}
in view of $a=a_2 \leq b = b_1$ according to \eqref{ab} and $\mu <1$.  As discussed in \cite{mu-pa-wei-2012}, by the orthogonality conditions the function $\Phi$ must vanish and a contradiction arises in this case.  If $p_0 \in \Pi \setminus \Pi_{i_0}$,  setting
$$M:=\Lambda U_1^{a_1-1} (p_0) \hbox{ if }i_0=1,\: M:= \Lambda U_2^{b_2-1} (p_0) \hbox{ if }i_0=2,$$
by $a_2<a_1+1,b_2+1$ according to \eqref{ab},  \eqref{1819}, \eqref{expansion} and \eqref{1217}  observe that 
\begin{eqnarray*}
\Lambda |\partial_t A(U_1,U_2)| (\cdot+p_0) =\mathcal O(\Lambda e^{- \mu a_1 \ell}) =\mathcal O(\Lambda e^{- \mu (a_2-1)\ell}) \to 0 && \hbox{if }i_0=1\\ 
\Lambda|\partial_s B(U_1,U_2)| (\cdot+p_0) =\mathcal O( \Lambda e^{- \mu b_2 \ell}) =\mathcal O( \Lambda e^{- \mu (a_2-1)\ell}) \to 0 && \hbox{if }i_0=2
\end{eqnarray*}
and 
$$\left\{ \begin{array}{ll}
\displaystyle\frac{\Lambda U_1^{a_1-1} (\cdot+p_0)}{M}&\hbox{if }i_0=1\\\\
\displaystyle\frac{\Lambda U_2^{b_2-1} (\cdot+p_0)}{M}&\hbox{if }i_0=2 \end{array}\right. \to G_\infty>0$$ 
do hold  uniformly on compact sets and $M\to M_\infty\in [0,+\infty]$,  up to a subsequence. We apply the Ascoli-Arzelà's theorem to $ \phi_{i_0} (\cdot+p_0)$ if $M_\infty<\infty$ or $ \phi_{i_0} (\frac{\cdot}{\sqrt M}+p_0)$ if $M_\infty=+\infty$ to deduce, up to a subsequence,  the $L^\infty_{loc}-$convergence to a non-trivial bounded solution of
$$-\Delta \Phi+H_\infty \Phi=0 \hbox{ in } \mathbb{R}^N,$$ 
where
$$H_\infty=\left\{\begin{array}{ll} 1+a_1 M_\infty G_\infty U^{a_2}(0) \hbox{ if }M_\infty<+\infty \hbox{ or }a_1 G_\infty U^{a_2}(0)
\hbox{ if }M_\infty=\infty & \hbox{ when }i_0=1\\ 
1+b_2 M_\infty G_\infty  U^{b_1}(0) \hbox{ if }M_\infty<+\infty \hbox{ or }b_2 G_\infty U^{b_1}(0)
\hbox{ if }M_\infty=\infty & \hbox{ when }i_0=2 . \end{array} \right. $$
Since $H_\infty>0$,  bounded solutions are necessarily trivial and a contradiction arises in case $p_0 \in \Pi \setminus \Pi_{i_0}$ too.  Estimate \eqref{1550} has been established.\\

Denoted as $(-\Delta+1)_{\mathcal H_i}^{-1}(g)$,  the minimizer of
$$\frac{1}{2}\int_{\mathbb{R}^N}|\nabla \phi|^2+\frac{1}{2}\int_{\mathbb{R}^N} \phi^2- \int_{\mathbb{R}^N}g \phi, \quad \phi \in \mathcal H_i,$$
is the unique symmetric solution of $-\Delta \phi+\phi=g+\mathscr M_ i({\bf h})$ for $g \in L^\infty_{symm}(\mathbb{R}^N)$ with $\|g\|_*<+\infty$.  Equation \eqref{1520} is equivalent to find $\phi \in \mathcal H$ solving
\begin{equation} \label{1643}
(\hbox{Id}+\mathcal K)(\phi)=\tilde g
\end{equation}
where $\tilde g_i=(-\Delta+1)_{\mathcal H_i}^{-1}(g_i)$ and 
\begin{eqnarray*}
\mathcal K_1(\phi)&=&(-\Delta+1)_{\mathcal H_1}^{-1}[-f'(U_1) \phi_1+\Lambda \partial_s A(U_1,U_2)\phi_1+ \Lambda \partial_t A(U_1,U_2)\phi_2]\\
\mathcal K_2(\phi)&=&(-\Delta+1)_{\mathcal H_2}^{-1}[-f'(U_2) \phi_2+\Lambda \partial_s B(U_1,U_2)\phi_1+\Lambda \partial_t B(U_1,U_2)\phi_2].
\end{eqnarray*}
Since $\mathcal K:\mathcal H \to \mathcal H$ is a compact operator by the exponential decay of $U_1$ and $U_2$,  the Fredholm's theory provides unique solvability for \eqref{1643} because $(\hbox{Id}+\mathcal K)(\phi)=0$ has only the trivial solution in view of \eqref{1550}. The unique solution $\phi \in \mathcal H$ of \eqref{1520} satisfies the estimate \eqref{1835} for the corresponding ${\bf h}$ in view of \eqref{1550}.

\end{proof}

\subsection{The error size}

\begin{lm}\label{error}
Assume \eqref{d5} and $0<\eta<\min\{c, {\mathfrak p}-1,1\}$ with $c$ given in \eqref{557}. There exists $\Lambda_0>0$  such that for any $\Lambda\ge \Lambda_0$ there holds
\begin{equation} \label{1201}
\|\mathscr E\|_* =\mathcal O \Big(  \Lambda e^{-(c-\eta)\mu \ell}+e^{- (\min\{{\mathfrak p}-1,1\}-\eta) \ell}\Big).
\end{equation}
\end{lm}
\begin{proof}
Let us estimate  the $*-$norm of  $\mathscr E $, where
$$\mathscr E_1=-\Delta U_1+U_1-U_1^{{\mathfrak p}}+\Lambda U_1^{a_1}U_2^{a_2}, \mathscr E_ 2=-\Delta U_2+U_2-U_2^{{\mathfrak p}}+\Lambda U_1^{b_1}U_2^{b_2}.$$
Since
$$\|\Lambda U_1^{a_1}U_2^{a_2}\|_*+\| \Lambda U_1^{b_1}U_2^{b_2}\|_* =\mathcal O( \Lambda e^{-(c-\eta) \mu \ell})$$
 in view of \eqref{1217} and \eqref{key}, it is enough to estimate
\begin{eqnarray*}
 \|-\Delta U_i+U_i-U_i^{{\mathfrak p}}\|_*&=&\Big \| \sum_{p \in\Pi_i} U^{{\mathfrak p}}(x-p)-\Big(\sum_{p\in\Pi_i} U (x-p)\Big)^{{\mathfrak p}}\Big\|_* \\
&=& \mathcal O \Big( \sum_{p,q\in\Pi_i\atop p\not=q} \Big\|  U^{{\mathfrak p}-1}(x-p)U(x-q)\Big\|_*\Big)  =\mathcal O (e^{- (\min\{{\mathfrak p}-1,1\} -\eta) \rho_i})
\end{eqnarray*}
in view of \eqref{710},  where $\rho_1$ and $\rho_2$ are given in \eqref{1217}.  Since $2\mu \cos \frac{\pi}{k}>
\frac{2 \cos \frac{\pi}{k}}{a-b+2\cos \frac{\pi}{k}} \geq 1$ in view of \eqref{d5} and $a\leq b$,  we have that $\rho_2\geq \rho_1$ and then the validity of \eqref{1201} follows.
\end{proof}
\subsection{The nonlinear projected problem}
In this section we will solve the nonlinear problem (see \eqref{L}-\eqref{Q} and \eqref{MMM})
\begin{equation}\label{pro-pro}
\mathscr L(\phi)=-\mathscr E-\mathscr Q(\phi)+\mathscr M({\bf h})\ \hbox{in}\ \mathbb{R}^N\end{equation}
provided  $\Lambda$ is large enough. \\

\begin{prop}\label{nonlinear}  Assume \eqref{d5} and let $\eta>0$ be suitably small.  There exists $\Lambda_0>0$  such that for any $\Lambda\ge \Lambda_0$ 
there exist $\phi$ satisfying \eqref{orto1}-\eqref{orto2} and ${\bf h}$ solving \eqref{pro-pro}. Moreover,  $\phi$ and ${\bf h}$ depend continuously on the parameters $\alpha_j,$ $\beta_h$ and $\gamma_h$ and satisfy
\begin{equation} \label{1551}
\|{\bf h}\|+\|\phi\|_* =\mathcal O\Big(    \Lambda e^{-(c-\eta)\mu \ell}+e^{- (\min\{{\mathfrak p}-1,1\}-\eta) \ell}\Big).
\end{equation}
\end{prop}
\begin{proof}
Denoting $\phi=\mathscr L^{-1}(g)$ in Proposition \ref{linear},  we can rewrite \eqref{pro-pro} as the fixed point problem $\phi=\mathcal T(\phi)$, where $\mathcal T(\phi):=-\mathscr L^{-1} [\mathscr E+\mathscr Q(\phi)]$. We apply  a standard contraction mapping argument to $\mathcal T$ on $B_M:=\left\{ \phi \in L^{\infty,*}_{symm} : \|\phi \|_*\leq M \epsilon \right\} $,  where $M$ is large and $\epsilon=\Lambda e^{-(c-\eta)\mu \ell}+e^{- (\min\{{\mathfrak p}-1,1\}-\eta) \ell}.$ Choosing $\eta>0$ sufficiently small, by \eqref{expansion} notice that
\begin{equation} \label{0845}
\Lambda (\epsilon^{c-1} +e^{-(c-2)\mu \ell} \epsilon)\rightarrow 0
\end{equation}
provided $\mu<\frac{1}{a+1-c}$ and 
\begin{equation} \label{0945}
 \mu<\frac{1}{a}+\frac{c-1}{a} \min\{{\mathfrak p}-1,1\}, \ \frac{1}{a+2-c}+\frac{1}{a+2-c} \min\{{\mathfrak p}-1,1\}.
\end{equation}
Observe that \eqref{0945} is automatically true when ${\mathfrak p}\geq 2 $ under assumption \eqref{d5}. \\
 
First of all,  by \eqref{Q} we have that
\begin{equation}\label{Q1}
\begin{aligned}
\mathscr Q_1(\phi)=&- \underbrace{\Big[ (  U_1+\phi_1)_+^{{\mathfrak p}}- U_1^{{\mathfrak p}}-{\mathfrak p}
U_1^{{\mathfrak p}-1}\phi_1 \Big]}_{I}\\
&+\underbrace{\Lambda \Big[|U_1+\phi_1|^{a_1-1}(U_1+\phi_1)-U_1 ^{a_1}-a_1 U_1^{a_1-1}\phi_1 \Big] |U_2+\phi_2|^{a_2}}_{II}\\
&+\underbrace{\Lambda  \Big[ |U_2+\phi_2|^{a_2}-U_2 ^{a_2}-a_2 U_2^{a_2-1}\phi_2 \Big] U_1^{a_1}}_{II}\\
&+\underbrace{
a_1\Lambda U_1^{a_1-1}\Big[|U_2+\phi_2|^{a_2}-U_2^{a_2}\Big] \phi_1}_{II}.
 \end{aligned}\end{equation}
Since $(t_1+t_2)_+^{\mathfrak p}=(t_1)_+^{\mathfrak p}+\gamma (t_1)_+^{{\mathfrak p}-1}t_2+\mathcal O(t_2^{\mathfrak p})+\underbrace{\mathcal O(t_1^{{\mathfrak p}-2}t_2^2)}_{\hbox{if }{\mathfrak p}\geq 2}$, we have that
\begin{equation}\label{Q1p}
I=\mathcal O(\phi_1^{\mathfrak p})+\underbrace{\mathcal O(U_1^{{\mathfrak p}-2}\phi_1^2)}_{\hbox{if }{\mathfrak p}\geq 2}
\end{equation} 
and then $\| I \|_*=o(\epsilon)$ uniformly for $\phi \in B_M$ since $\epsilon \to 0$ and ${\mathfrak p}>1$. Since 
\begin{eqnarray*}
&& |t_1+t_2|^{a_1-1}(t_1+t_2)-|t_1|^{a_1-1}t_1-a_1 |t_1|^{a_1-1}t_2=\mathcal O(|t_2|^{a_1})+\underbrace{\mathcal O(|t_1|^{a_1-2}t_2^2)}_{\hbox{if }a_1 \geq 2}\\
&& |t_1+t_2|^{a_2}-|t_1|^{a_2}=\mathcal  O(|t_1|^{a_2-1}|t_2|+|t_2|^{a_2})=a_2 |t_1|^{a_2-2}t_1 t_2 +\mathcal O(|t_2|^{a_2})+\underbrace{\mathcal O(|t_1|^{a_2-2}t_2^2)}_{\hbox{if }a_2\geq 2},
\end{eqnarray*}
we deduce that
\begin{eqnarray} \label{Q1pp}
II &=&\Lambda \mathcal O\Big( U_2^{a_2} |\phi_1|^{a_1}+ U_1^{a_1} |\phi_2|^{a_2}+ |\phi_2|^{a_2} |\phi_1|^{a_1}+U_1^{a_1-1}U_2^{a_2-1} |\phi_1| |\phi_2|+U_1^{a_1-1} |\phi_1| |\phi_2|^{a_2}\Big)\nonumber \\
&&+\Lambda \mathcal O \Big( \underbrace{ U_1^{a_1-2}U_2^{a_2}  |\phi_1|^2+U_1^{a_1-2} |\phi_2|^{a_2}  |\phi_1|^2}_{\hbox{if }a_1 \geq 2}
+ \underbrace{ U_1^{a_1} U_2^{a_2-2} |\phi_2|^2}_{\hbox{if }a_2 \geq 2}\Big)
\end{eqnarray}
and then
$$\|II\|_*=\Lambda \mathcal O( \|\phi_1\|_*^{a_1}+\|\phi_2 \|_*^{a_2}+e^{-(c-2) \mu \ell} \|\phi_1\|_*^2
+e^{-(c-2)  \mu \ell} \|\phi_2\|_*^2)=o(\epsilon)$$ 
uniformly for $\phi \in B_M$ in view of \eqref{0845}.  Collecting the estimates on $I$ and $II$ we have that $\|\mathscr Q_1(\phi)\|_*=o(\epsilon)$ uniformly for $\phi \in B_M$ and the same is true for $\mathscr Q_2(\phi)$ by a similar argument.  Since $\|\mathscr L^{-1} (\mathscr E)\|_* =O( \|\mathscr E)\|_*)\leq C \epsilon$ in view of Proposition \ref{linear}, we can take $M=2C$ and then $\mathcal T:B_M\to B_M$ in view of $\|\mathcal T(\phi)\| \leq C\epsilon+o(\epsilon) \leq M\epsilon$ for all $\phi \in B_M$. \\

Next,  arguing as above,  for the quantities
$$
\begin{aligned}
&\mathscr Q_1(\phi)-\mathscr Q_1(\bar\phi)=(U_1+\bar\phi_1)_+^{\mathfrak p}  - ( U_1+\phi_1)_+^{\mathfrak p}+{\mathfrak p} U_1^{{\mathfrak p}-1}(\phi_1-\bar\phi_1)\\
& +\Lambda\Big[|U_1+\phi_1|^{a_1-1}(U_1+\phi_1) |U_2+\phi_2|^{a_2}-|U_1+\bar \phi_1|^{a_1-1}(U_1+\bar \phi_1) |U_2+\bar \phi_2|^{a_2}\\
&-a_1 U_1^{a_1-1} U_2^{a_2}(\phi_1-\bar\phi_1) -a_2 U_1^{a_1} U_2^{a_2-1}(\phi_2-\bar\phi_2)\Big]  \\
 \end{aligned}$$
 and 
$$\begin{aligned}
&\mathscr Q_2(\phi)-\mathscr Q_2(\bar\phi)=(U_2+\bar\phi_2)_+^{\mathfrak p}-(U_2+\phi_2)_+^{\mathfrak p}+{\mathfrak p}U_2^{{\mathfrak p}-1}(\phi_2-\bar\phi_2)\\
& +\Lambda\Big[|U_1+\phi_1^{b_1}|U_2+\phi_2|^{b_2-1}(U_2+\phi_2)-|U_1+\bar \phi_1^{b_1}|U_2+\bar \phi_2|^{b_2-1}(U_2+\bar \phi_2)\\
&-b_1  U_1^{b_1-1} U_2^{b_2}(\phi_1-\bar\phi_1)-b_2 U_1^{b_1} U_2^{b_2-1} (\phi_2-\bar\phi_2)\Big]  \\
 \end{aligned}$$
we can prove $\|\mathcal T(\phi)-\mathcal T(\bar \phi)\|=O(\| \mathscr Q(\phi)- \mathscr Q(\bar\phi)\|_*)  \leq \mathtt L \|\phi-\bar\phi\|_*$ for any $\phi,  \bar\phi \in B$, for some $\mathtt L\in(0,1)$, in view of \eqref{0845} and the mean value theorem. Hence,  we can find $\Lambda_0>0$ large such that the map $\mathcal T:B_M\to B_M$ is a contraction for any $\Lambda\ge \Lambda_0$. By the contraction mapping principle a unique fixed point $\phi \in B_M$ for $\mathcal T$ exists, yielding a solution $\phi$ and ${\bf h}$ to \eqref{pro-pro} satisfying \eqref{orto1}-\eqref{orto2} and \eqref{1551}. The continuous dependence of the solution $\phi,\ {\bf h} $ upon $\alpha_j, \beta_j$ and $\gamma_j$ is standard (e.g.  \cite{mu-pa-wei-2012}).
\end{proof}

\section{The reduced problem} \label{sec1122}
In this section we aim to obtain expansions for the projection of the error $\mathscr E $ onto the kernel. 
We  decompose $\mathscr E$ in \eqref{E} into a {\em inner} part $\mathscr I$ and an {\em outer} part $\mathscr O$ as follows
\begin{equation} \label{1026}
\mathscr E_1:=\underbrace{\sum_{p\in\Pi_1}U^{{\mathfrak p}}(x-p)-U_1^{{\mathfrak p}}}_{={\mathscr I}_1}+\underbrace{\Lambda U_1^{a_1}U_2^{a_2}}_{={\mathscr O}_1},\quad \mathscr E_2:=\underbrace{\sum_{p\in\Pi_2}U^{{\mathfrak p}}(x-p)-U_2^{{\mathfrak p}}}_{={\mathscr I}_2}+\underbrace{\Lambda U_1^{b_1}U_2^{b_2}}_{={\mathscr O}_2}
\end{equation}
in view of \eqref{bubble} and \eqref{u1}-\eqref{u2}. In the following we agree that  
$D_\Lambda:=D_\Lambda( \boldsymbol \alpha, \boldsymbol \beta, \boldsymbol \gamma)$   
and  $Q_\Lambda:=Q_\Lambda(\boldsymbol \alpha, \boldsymbol \beta, \boldsymbol \gamma)$ denote continuous expressions which are 
linear and quadratic, respectively, in $\boldsymbol \alpha, \boldsymbol \beta, \boldsymbol \gamma$ and $\delta>0$ is a small number independent of 
$\boldsymbol \alpha, \boldsymbol \beta, \boldsymbol \gamma$. They might vary from line to line.\\

Recalling the definition of $\mathtt t$ and $\mathtt n$ in \eqref{vect},  introduce their $x_2-$reflections 
\begin{equation}\label{tstar}
\mathtt t^*:=-\sin\frac\pi k \mathtt e_1-\cos\frac\pi k \mathtt e_2,\quad \mathtt n^*:=\cos\frac \pi k \mathtt e_1-\sin\frac\pi k \mathtt e_2, \end{equation}
and notice that
\begin{equation}\label{rotate}
R_{-k}\mathtt t=-\mathtt t^*, \quad R_{-k}\mathtt n=\mathtt n^*.
 \end{equation}

\subsection{The outer error}
First let us make expansions for the {\em outer error} $\mathscr O$. As already explained in the Introduction,  these estimates are new and delicate because of the non-linear interaction among peaks.

\begin{lm}\label{new-1} 
 ${}$\\
As $\Lambda \to +\infty$ the first component ${\mathscr O}_1$ satisfies
\begin{eqnarray*}
\frac{\int\limits_{\mathbb R^N} {\mathscr O}_1(x) \nabla U(x-y_i) dx}{ \Psi^{(0)}(\ell)}  &=& \mathcal O\Big(  e^{- \delta \ell}\Big)\\
\frac{\int\limits_{\mathbb R^N} {\mathscr O}_1(x) \nabla U(x-y_{m+1})  dx}{\Lambda \Psi^{(1)}(\bar \ell) }&=& [1
-  a_2 ( \beta_1+\alpha_{m+1} \sin \frac{\pi}{k})+\mathfrak c_2 \gamma_1] \mathtt e_1\\
&&+o(1)D_\Lambda+\mathcal O(e^{- \delta \bar\ell^\frac{1}{4}}+\bar \ell^K Q_\Lambda^\frac{\theta}{2}) \\
\frac{\int_{\mathbb{R}^N} {\mathscr O}_1(x) \nabla U(x -z_{2h})  dx}{\Lambda \Psi^{(1)}(\bar \ell) } &=& \mathfrak c_3 (\beta_{2h+1}+\beta_{2h-1}-2\beta_{2h}) \mathtt t+
\mathfrak c_4 (\gamma_{2h+1}+\gamma_{2h-1}-2 \gamma_{2h}) \mathtt n\\
&&+o(1)D_\Lambda+\mathcal O( e^{-\delta\bar\ell^\frac{1}{4}}+\bar \ell^K Q_\Lambda^\frac{\theta}{2})
\end{eqnarray*}
for  $i=1, \dots, m$ and $h=1, \dots, n-1$,  where $K>0$ is large and  (see \eqref{psiuno})
$$\Psi^{(1)}(\bar\ell)=\bar \ell^{-\frac{N-1}{2}a_2} e^{-a_2 \bar \ell }(\mathfrak c_1+o(1)),\ \mathfrak c_1>0.$$

As $\Lambda \to +\infty$ the second component ${\mathscr O}_2$ satisfies
\begin{eqnarray*}
\frac{\int\limits_{\mathbb R^N}{\mathscr O}_2(x) \nabla U(x -z_{1}) dx}{\Lambda\Psi^{(2)}(\bar\ell)} &=& \mathfrak c_5 (\beta_{2}-2\beta_1-\alpha_{m+1}\sin\frac\pi k) \mathtt t+\mathfrak c_6 (\gamma_2-2\gamma_1) \mathtt n\\
&&+o(1)D_\Lambda+\mathcal O(e^{-\delta \bar\ell^\frac{1}{4}}+\bar \ell^K Q_\Lambda^\frac{\theta}{2})\\
\frac{\int\limits_{\mathbb R^N}  {\mathscr O}_2(x)  \nabla U(x -z_{2h-1})   dx}{\Lambda\Psi^{(2)}(\bar\ell)} &=& \mathfrak c_5 (\beta_{2h}+\beta_{2h-2}-2\beta_{2h-1}) \mathtt t+\mathfrak c_6 (\gamma_{2h}+\gamma_{2h-2}-2 \gamma_{2h-1}) \mathtt n\\
&&+o(1)D_\Lambda+\mathcal O(e^{-\delta \bar\ell^\frac{1}{4}}+\bar \ell^K Q_\Lambda^\frac{\theta}{2})
\end{eqnarray*}
for $h=2,\dots, n$,  where $\Psi^{(2)}(\bar\ell)=\bar \ell^{-\frac{N-1}{2}b_1} e^{-b_1 \bar \ell }$ (see \eqref{psidue}). Here $\mathfrak c_3, \dots, \mathfrak c_6 \not=0$ and $\theta:=\min\{a_2,2\}>1.$  
 
\end{lm}
\begin{proof}
First of all,  in the Appendix we show that
\begin{equation}\label{us1}
\int_{\mathbb{R}^N} {\mathscr O}_1(x)\nabla U(x -s) dx =\Lambda {\boldsymbol\Psi}^{(1)} (s)+ \mathcal O(\Lambda e^{-(1+\delta)a_2\bar\ell}) \end{equation}
when $s\in\Pi_1$ and
\begin{equation}\label{us2}
\int_{\mathbb{R}^N} {\mathscr O}_2(x)\nabla U(x -s) dx= \Lambda {\boldsymbol\Psi}^{(2)} (s)+ \mathcal O( \Lambda e^{-(1+\delta)b_1\bar\ell})
\end{equation}
when $s\in\Pi_2$. \\


Now, let us estimate the  leading term ${\boldsymbol\Psi}^{(i)}(s),$ for $s\in \Pi_i$ and $i=1,2.$
\\

{\em The first component.}\\
If \fbox{$s=y_1,\dots,y_m$} then
$|s-q|\geq (1+\delta) \ell$  for any $q\in \Pi_2^s$ and  ${\boldsymbol\Psi}^{(1)}(y_i)= \mathcal O(e^{-(1+\delta) a_2\ell})$.\\

The case   \fbox{$s=z_{2h} $}, 
$0\leq h \leq n-1,$  is much more involved. Since $\Pi_2^s=\{z_{2h-1},z_{2h+1} \}$ the leading term reduces to
$${\boldsymbol\Psi}^{(1)} (z_{2h})= \int_{\mathbb{R}^N} \Big[U(x-p_1)+U(x-p_2) \Big]^{a_2} U^{a_1}\nabla U,$$
where $p_1=z_{2h-1}-z_{2h}$ and $p_2=z_{2h+1}-z_{2h}$. First of all,  let us perform the following Taylor expansion of ${\boldsymbol\Psi}^{(1)}$ around $z^*_{2h}$:
\begin{equation}\label{taylor} {\boldsymbol\Psi}^{(1)}(z_{2h}) ={\boldsymbol\Psi}^{(1)}(z^*_{2h})-a_2  \mathtt D {\boldsymbol\Psi}^{(1)} (z^*_{2h})+\mathcal O( e^{-a_2 \bar \ell-\delta \bar \ell^\frac{1}{4}} +\bar \ell^\theta e^{-a_2 \bar \ell} Q_\Lambda^\frac{\theta}{2})
\end{equation}
with
$$\begin{aligned}
{\boldsymbol\Psi}^{(1)}(z^*_{2h})&= \int_{\mathbb{R}^N} \left[U(x-p_1^*)+U(x-p_2^*)\right]^{a_2} U^{a_1}\nabla U, \\
\mathtt D {\boldsymbol\Psi}^{(1)} (z^*_{2h})&=\int_{\mathbb{R}^N}  [U(x-p_1^*)+U(x-p_2^*)]^{a_2-1}[\langle \nabla U(x-p_1^*), \mathfrak {d} p_1 \rangle+\langle \nabla U(x-p_2^*), \mathfrak {d} p_2 \rangle] U^{a_1}\nabla U,\end{aligned}
$$
where  $p_1^*=z^*_{2h-1}-z^*_{2h}$, $p_2^*=z^*_{2h+1}-z^*_{2h}$ and  $\mathfrak {d} p_i=p_i-p_i^*$. Hereafter, we will often use that $\bar \ell^m Q_\Lambda=\mathcal O(1)$ for any $m$, property which will be true a posteriori because of the exponential smallness of $\boldsymbol \alpha$,  $\boldsymbol \beta$ and $\boldsymbol \gamma$ in $\bar \ell$.\\

Since
\begin{eqnarray}\label{1617}
p_1=\left\{ \begin{array}{ll} (\bar \ell +\beta_1) \mathtt t^* +\bar \ell \gamma_1 \mathtt n^*-\alpha_{m+1}\mathtt e_1 &\hbox{if }h=0\\
-(\bar \ell +\beta_{2h}-\beta_{2h-1})\mathtt t +\bar \ell (\gamma_{2h-1}-\gamma_{2h}) \mathtt n &\hbox{if }h\geq 1 \end{array} \right. 
\end{eqnarray}
and
\begin{eqnarray}\label{1618}
p_2=\left\{\begin{array}{ll}(\bar \ell +\beta_1 )\mathtt t +\bar \ell \gamma_1 \mathtt n-\alpha_{m+1}\mathtt e_1&\hbox{if }h=0\\
(\bar \ell +\beta_{2h+1} -\beta_{2h})\mathtt t +\bar \ell (\gamma_{2h+1}-\gamma_{2h}) \mathtt n
&\hbox{if }h \geq 1 \end{array} \right.
\end{eqnarray}
in view of $z_{-1}=R_{-k}z_{2n-1}$, \eqref{symme} and \eqref{rotate},  then $p_1$ and $p_2$ satisfy 
\begin{eqnarray}
|p_1|= |p_2|=\bar \ell +\beta_1+\alpha_{m+1} \sin \frac{\pi}{k}+\mathcal O\left( \bar \ell Q_\Lambda\right) \label{1619}
\end{eqnarray}
if $h=0$ and
\begin{eqnarray}
 |p_1|=\bar \ell +\beta_{2h}-\beta_{2h-1}+\mathcal O\left(\bar \ell Q_\Lambda\right),
|p_2|=\bar \ell + \beta_{2h+1} -\beta_{2h}+\mathcal O\left(\bar \ell Q_\Lambda\right)
\label{1620}
\end{eqnarray}
if $h \ge1$.  Next, by a Taylor expansion around $p_i^*$:
$$U(x-p_i)=U(x-p_i^*)- \langle \nabla U(x-p_i^*), \mathfrak {d} p_i \rangle+\mathcal  O(|\mathfrak {d} p_i|^2 U(x-p_i^*)),\ i=1,2,$$
in view of \eqref{cpabg} and $e^{|x-q|-|x-r|}\leq e^{|r-q|}$,  we deduce that 
\begin{eqnarray*}
\Big[U(x-p_1)+U(x-p_2)\Big]^{a_2}&=& \Big[U(x-p_1^*)+U(x-p_2^*)\Big]^{a_2}\\
&&-a_2 \Big[U(x-p_1^*)+U(x-p_2^*)\Big]^{a_2-1} \sum_{j=1}^2 \langle \nabla U(x-p_j^*), \mathfrak {d} p_j \rangle\\
&&+\Big[U(x-p_1^*)+U(x-p_2^*)\Big]^{a_2} \mathcal  O(\bar \ell^\theta Q_\Lambda^\frac{\theta}{2})
\end{eqnarray*}
for $|x| \leq \bar \ell^\frac{1}{4}$ and $\theta=\min\{a_2,2\}$ thanks to \eqref{1617}-\eqref{1618}. Therefore, there holds
 \begin{align*}
 {\boldsymbol\Psi}^{(1)}(z_{2h})	 &=\int_{\{|x|\le\bar\ell^\frac14\}}[U(x-p_1^*)+U(x-p_2^*)]^{a_2} U^{a_1} \nabla U\\
& -a_2\int_{\{|x|\le\bar\ell^\frac14\}}
 [U(x-p_1^*)+U(x-p_2^*)]^{a_2-1}\sum_{j=1}^2  \langle \nabla U(x-p_j^*), \mathfrak {d} p_j \rangle
U^{a_1} \nabla  U\\
&+\int_{\{|x|\le\bar\ell^\frac14\}}[U(x-p_1^*)+U(x-p_2^*)]^{a_2} U^{a_1} \nabla U\mathcal  O(\bar \ell^\theta Q_\Lambda^\frac{\theta}{2})\\
&+\int_{\{|x|\ge\bar\ell^\frac14\}} \left[U(x-p_1)+U(x-p_2)\right]^{a_2}U^{a_1} \nabla U\\
&=   {\boldsymbol\Psi}^{(1)}(z^*_{2h})-a_2  \mathtt D{\boldsymbol\Psi}^{(1)}(z^*_{2h})+\mathcal O\Big(\sum_{j=1}^2 \int_{\{|x|\ge\bar\ell^\frac14\}}[U^{a_2}(x-p_j^*)+U^{a_2}(x-p_j)] U^{a_1}| \nabla U| \Big)
\\
&+\mathcal  O\Big (  \bar \ell^\theta Q_\Lambda^\frac{\theta}{2} \int_{\{|x|\le\bar\ell^\frac14\}}[U^{a_2}(x-p_1^*)+U^{a_2}(x-p_2^*)] U^{a_1}| \nabla U| \Big)\\
&= {\boldsymbol\Psi}^{(1)}(z^*_{2h})-a_2 \mathtt D {\boldsymbol\Psi}^{(1)}(z^*_{2h})+\mathcal O( e^{-a_2 \bar \ell-\delta \bar \ell^\frac{1}{4}} +\bar \ell^\theta e^{-a_2 \bar \ell} Q_\Lambda^\frac{\theta}{2})
\end{align*}
because  of \eqref{1619}-\eqref{1620} and
\begin{equation} \label{1148}
U^{a_2}(x-p) U^{a_1}(x)|\nabla  U(x)|=\mathcal O\left(\frac{e^{-a_2 |p|}  e^{-(a_1+1-a_2)|x|}}{(1+|x-p|)^{\frac{N-1}{2}a_2} (1+|x|)^\frac{(N-1)(a_1+1)}{2}}\right)
\end{equation}
in view of  \eqref{0849}  and \eqref{key}. Once \eqref{taylor} has been established,  we have to expand ${\boldsymbol\Psi}^{(1)}(z^*_{2h})$ and $ \mathtt D {\boldsymbol\Psi}^{(1)}(z^*_{2h})$ when $h=0$ and $h\ge1.$
\\


First of all, let us consider the case \fbox{$h=0$}.
Since $\mathtt t_1=(\mathtt t^*)_1$, $\mathtt t_2=-(\mathtt t^*)_2$, $(\mathfrak {d} p_1)_1=(\mathfrak {d} p_2)_1$ and $(\mathfrak {d} p_1)_2=-(\mathfrak {d} p_2)_2$ in view of \eqref{vect} and \eqref{tstar}-\eqref{rotate},  by oddness under the reflection $x \to (x_1,\dots,-x_i,\dots, x_N)$, $i \geq 2$, notice that
\begin{equation}\label{1604}
{\boldsymbol\Psi}^{(1)}(z^*_{0})=\Psi^{(1)}(\bar \ell) \mathtt e_1, \quad  \Psi^{(1)}(\bar \ell):=
\int_{\mathbb{R}^N} \left[U(x-\bar \ell \mathtt t^*)+U(x-\bar \ell \mathtt t)\right]^{a_2} U^{a_1} \partial_1 U,
\end{equation}
and $\mathtt D {\boldsymbol\Psi}^{(1)}(z^*_{0})=\mathtt D _0(\bar\ell)  \mathtt  e_1$ with
$$ \mathtt D _0(\bar\ell)= \int_{\mathbb{R}^N}  [U(x-\bar \ell \mathtt t^*)+U(x-\bar \ell \mathtt t)]^{a_2-1}
\left[\langle \nabla U(x-\bar \ell \mathtt t^*), \mathfrak {d} p_1 \rangle+\langle \nabla U(x-\bar \ell \mathtt t), \mathfrak {d} p_2 \rangle\right] U^{a_1}\partial_1 U.
$$
The coefficient $\Psi^{(1)}(\bar \ell)$
plays a crucial role in the balance condition at $s=z_0^*=y_{m+1}^*$ and its expansion will be   necessary for the appropriate choice of $\ell$ and $\bar \ell$. We claim that
\begin{equation}\label{psi1barell-2}\Psi^{(1)}(\bar \ell)=\bar \ell^{-\frac{N-1}{2}a_2} e^{-a_2 \bar \ell}\left(\mathfrak c_1+o(1)\right),\  \mathfrak c_1>0.\end{equation}
Indeed, by $|x-p|=|p|-\langle x,\frac{p}{|p|}\rangle+\mathcal O(\frac{|x|^2}{|p|})$ we get
\begin{equation} \label{1110}
\begin{aligned}&e^{-|x-p|} =e^{-|p|}  e^{\langle x,\frac{p}{|p|}\rangle}\left[1+\mathcal O\left(\frac{|x|^2}{|p|}\right)\right] \\
& \frac{x-p}{|x-p|}=-\frac{p}{|p|}+\frac{1}{|p|}\left[x-\left\langle x, \frac{p}{|p|}\right\rangle \frac{p}{|p|}\right]+\mathcal O\left(\frac{|x|^2}{|p|^2}\right)\end{aligned}
\end{equation}
for $|x|\leq \bar \ell^{\frac{1}{4}}$ and so
$$\begin{aligned} &\bar \ell^{\frac{N-1}{2}a_2} e^{a_2 \bar \ell} \Psi^{(1)}(\bar \ell)\\
&=\bar \ell^{\frac{N-1}{2}a_2} e^{a_2 \bar \ell}  \int_{B(0,\bar \ell^\frac{1}{4})} \left[U(x-\bar \ell \mathtt t^*)+U(x-\bar \ell \mathtt t)\right]^{a_2} U^{a_1}\partial_1 U +\mathcal O( \bar \ell^{\frac{N-1}{2}a_2} e^{-\delta \bar \ell^{\frac14}})  \\
&=\underbrace{c_N^{a_2} \int_{\mathbb{R}^N} \left[e^{\langle x,  \mathtt t^* \rangle} +e^{\langle x, \mathtt t \rangle} \right]^{a_2} U^{a_1}\partial_1 U}_{=:\mathfrak c_1>0} +o(1)
\end{aligned}$$
as $\bar \ell \to +\infty$ in view of \eqref{1148}, where $c_N=\displaystyle \lim_{|x|\to +\infty} |x|^\frac{N-1}{2} e^{|x|} U(x)$.
It is important to point out that $\mathfrak c_1>0$ as it easily follows through an integration by parts:
 $$\int_{\mathbb{R}^N} \left[e^{\langle x,  \mathtt t^* \rangle} +e^{\langle x, \mathtt t \rangle} \right]^{a_2} U^{a_1}\partial_1 U= \frac{a_2}{a_1+1} \sin \frac{\pi}{k} \int_{\mathbb{R}^N}  \left[e^{\langle x,  \mathtt t^* \rangle} +e^{\langle x, \mathtt t \rangle} \right]^{a_2}U^{a_1+1}>0.$$ 

Next, we  expand the coefficient of the  linear term $ \mathtt D {\boldsymbol\Psi}^{(1)}  (z^*_{0})$. We have
$$\begin{aligned} 
&   \mathtt D _0(\bar\ell)\\  &=
\gamma_1 \int_{B(0,\bar \ell^\frac{1}{4})}  [U(x-\bar \ell \mathtt t^*)+U(x-\bar \ell \mathtt t)]^{a_2-1} [U'(x-\bar \ell \mathtt t^*) \langle x,\mathtt n^* \rangle+U'(x-\bar \ell \mathtt t)\langle x,\mathtt n \rangle] U^{a_1}\partial_1 U\\
&-\left(\beta_1+\alpha_{m+1} \sin \frac{\pi}{k}\right) \int_{B(0,\bar \ell^\frac{1}{4})}  [U(x-\bar \ell \mathtt t^*)+U(x-\bar \ell \mathtt t)]^{a_2-1}[U'(x-\bar \ell \mathtt t^*)+U'(x-\bar \ell \mathtt t)] U^{a_1}\partial_1 U\\
&+o(\bar \ell^{-\frac{N-1}{2}a_2} e^{-a_2 \bar \ell })D_\Lambda
\end{aligned}$$
because of \eqref{1148} and 
$$\langle \frac{x-\bar \ell \mathtt t^*}{|x-\bar \ell \mathtt t^*|},\mathfrak {d} p_1\rangle-\gamma_1 \langle x, \mathtt n^* \rangle,\  \langle \frac{x-\bar \ell \mathtt t}{|x-\bar \ell \mathtt t|}, \mathfrak {d} p_2 \rangle-\gamma_1 \langle x,\mathtt n \rangle=-\left(\beta_1+\alpha_{m+1} \sin \frac{\pi}{k}\right) +o(1)D_\Lambda$$
for $|x|\leq \bar \ell^{\frac{1}{4}}$ thanks to \eqref{1110}. Hence there holds
$$\begin{aligned} 
& \bar \ell^{\frac{N-1}{2}a_2} e^{a_2 \bar \ell } \mathtt D  {\boldsymbol\Psi}^{(1)}(z^*_{0})\\ &= \left(\beta_1+\alpha_{m+1} \sin \frac{\pi}{k}\right) c_N^{a_2}  \int_{\mathbb{R}^N}  \left[e^{\langle x,  \mathtt t^* \rangle} +e^{\langle x, \mathtt t \rangle} \right]^{a_2}U^{a_1}  \partial_1  U \mathtt e_1\\
&-\gamma_1 \underbrace{c_N^{a_2}  \int_{\mathbb{R}^N} \left[e^{\langle x,  \mathtt t^* \rangle} +e^{\langle x, \mathtt t \rangle} \right]^{a_2-1}[e^{\langle x,  \mathtt t^* \rangle} \langle x, \mathtt n^* \rangle+e^{\langle x, \mathtt t \rangle} \langle x,\mathtt n \rangle] U^{a_1} \partial_1 U}_{=\frac{\mathfrak c_1 \mathfrak c_2}{a_2}}\mathtt e_1\\ &+o(1)D_\Lambda \\
&=\left[\mathfrak c_1 \left(\beta_1+ \alpha_{m+1} \sin \frac{\pi}{k}\right)-  \frac{\mathfrak c_1 \mathfrak c_2 }{a_2} \gamma_1 \right]\mathtt e_1+o(1)D_\Lambda\end{aligned}
$$
in view of \eqref{1110}.
 
Next, we consider the case \fbox{$h\ge1$}.
 Decomposing $\nabla =\mathtt t \partial_{\mathtt t}+\mathtt n \partial_{\mathtt n} +\displaystyle \sum_{j=3}^N e_j \partial_{e_j}$,  we have that
$$\begin{aligned}
  {\boldsymbol\Psi}^{(1)}(z^*_{2h})&= \int_{\mathbb{R}^N} \left[U(x+\bar \ell \mathtt t)+U(x-\bar \ell \mathtt t)\right]^{a_2} U^{a_1} \nabla  U\\ &
=\mathtt t \int_{\mathbb{R}^N} \left[U(x+\bar \ell \mathtt t)+U(x-\bar \ell \mathtt t)\right]^{a_2} U^{a_1} \partial_{\mathtt t}U  \\ &\quad +\mathtt n \int_{\mathbb{R}^N} \left[U(x+\bar \ell \mathtt t)+U(x-\bar \ell \mathtt t)\right]^{a_2}U^{a_1} \partial_{\mathtt n} U
\\ &\quad+\sum_{j=3}^N e_j \int_{\mathbb{R}^N} \left[U(x+\bar \ell \mathtt t)+U(x-\bar \ell \mathtt t)\right]^{a_2}U^{a_1} \partial_{e_j} U
=0
\end{aligned}$$
 by oddness of $i-$th term above under the reflection $x \to T_i(x)$, $i=1,\dots,N$, where
$$T_i(x)=\left\{\begin{array}{ll} - \langle x,  \mathtt t\rangle \mathtt t+\langle x,  \mathtt n\rangle \mathtt n+\displaystyle \sum_{j=3}^N \langle x, e_j \rangle e_j & \hbox{if }i=1\\
\langle x,  \mathtt t\rangle \mathtt t-\langle x,  \mathtt n\rangle \mathtt n+\displaystyle \sum_{j=3}^N \langle x, e_j \rangle e_j&\hbox{if }i=2\\
\langle x,  \mathtt t\rangle \mathtt t+\langle x,  \mathtt n\rangle \mathtt n+\displaystyle \sum_{j \not= i} \langle x, e_j \rangle e_j-\langle x, e_i \rangle e_i &\hbox{if }i\geq 3. \end{array} \right.$$
Next, we expand the derivative  $ \mathtt D {\boldsymbol\Psi}^{(1)}(z^*_{2h})$. By exploiting all the symmetries we can write
$$\begin{aligned}
&\mathtt D {\boldsymbol\Psi}^{(1)}(z^*_{2h})\\
&= \mathtt t \int_{\mathbb{R}^N}  [U(x+\bar \ell \mathtt t)+U(x-\bar \ell \mathtt t)]^{a_2-1}[\langle \nabla U(x+\bar \ell  \mathtt t), \mathfrak {d} p_1 \rangle+\langle \nabla U(x-\bar \ell \mathtt t), \mathfrak {d} p_2 \rangle] U^{a_1} \partial_{\mathtt t} U\\
&+\mathtt n \int_{\mathbb{R}^N}  [U(x+\bar \ell \mathtt t)+U(x-\bar \ell \mathtt t)]^{a_2-1}[\langle \nabla U(x+\bar \ell  \mathtt t), \mathfrak {d} p_1 \rangle+\langle \nabla U(x-\bar \ell \mathtt t), \mathfrak {d} p_2 \rangle] U^{a_1} \partial_{\mathtt n} U\\
& =-(\beta_{2h}-\beta_{2h-1}) \mathtt t \int_{\mathbb{R}^N}  [U(x+\bar \ell \mathtt t)+U(x-\bar \ell \mathtt t)]^{a_2-1}
\partial_{\mathtt t} U(x+\bar \ell  \mathtt t)  U^{a_1} \partial_{\mathtt t} U\\
&+(\beta_{2h+1}-\beta_{2h}) \mathtt t \int_{\mathbb{R}^N}  [U(x+\bar \ell \mathtt t)+U(x-\bar \ell \mathtt t)]^{a_2-1} \partial_{\mathtt t} U(x-\bar \ell \mathtt t) U^{a_1} \partial_{\mathtt t} U\\
&+\bar \ell (\gamma_{2h-1}-\gamma_{2h}) \mathtt n \int_{\mathbb{R}^N}  [U(x+\bar \ell \mathtt t)+U(x-\bar \ell \mathtt t)]^{a_2-1} \partial_{\mathtt n} U(x+\bar \ell  \mathtt t)  U^{a_1} \partial_{\mathtt n} U\\
&+\bar \ell (\gamma_{2h+1}-\gamma_{2h})  \mathtt n \int_{\mathbb{R}^N}  [U(x+\bar \ell \mathtt t)+U(x-\bar \ell \mathtt t)]^{a_2-1}\partial_{\mathtt n} U(x-\bar \ell \mathtt t) U^{a_1} \partial_{\mathtt n} U\\
& =(\beta_{2h+1}+\beta_{2h-1}-2\beta_{2h}) \mathtt t \underbrace{\int_{\mathbb{R}^N}  [U(x+\bar \ell \mathtt t)+U(x-\bar \ell \mathtt t)]^{a_2-1} \partial_{\mathtt t} U(x-\bar \ell \mathtt t) U^{a_1} \partial_{\mathtt t} U}_{=\mathtt D_{\mathtt t}(\bar\ell)}\\
&+\bar \ell (\gamma_{2h+1}+\gamma_{2h-1}-2\gamma_{2h})  \mathtt n \underbrace{\int_{\mathbb{R}^N}  [U(x+\bar \ell \mathtt t)+U(x-\bar \ell \mathtt t)]^{a_2-1}\partial_{\mathtt n} U(x-\bar \ell \mathtt t)  U^{a_1} \partial_{\mathtt n} U}_{=\mathtt D_{\mathtt n}(\bar\ell)}
\end{aligned}$$
in view of $$\mathfrak {d} p_1=-(\beta_{2h}-\beta_{2h-1})\mathtt t +\bar \ell (\gamma_{2h-1}-\gamma_{2h}) \mathtt n,\ \mathfrak {d} p_2=(\beta_{2h+1} -\beta_{2h})\mathtt t +\bar \ell (\gamma_{2h+1}-\gamma_{2h}) \mathtt n.$$
By symmetry and through some integration by parts the coefficients $\mathtt D_{\mathtt t}(\bar\ell)$ and $\mathtt D_{\mathtt n}(\bar\ell)$ have the following expansions:
$$\begin{aligned}
&\bar \ell^{\frac{N-1}{2}a_2} e^{a_2 \bar \ell }  \mathtt D_{\mathtt t}(\bar \ell)\\
&=c_N^{a_2}\int_{\mathbb{R}^N}  (e^{-\langle x,\mathtt t \rangle}+e^{\langle x,\mathtt t\rangle})^{a_2-1}e^{\langle x,\mathtt t\rangle}  U^{a_1} \partial_{\mathtt t} U+o(1)\\
&= \frac{c_N^{a_2}}{2} \int_{\mathbb{R}^N}  (e^{-\langle x,\mathtt t \rangle}+e^{\langle x,\mathtt t\rangle})^{a_2-1}
(e^{\langle x,\mathtt t\rangle}-e^{-\langle x,\mathtt t\rangle})U^{a_1}  \partial_{\mathtt t}  U+o(1)\\
&=- \frac{c_N^{a_2}}{2(a_1+1)} \int_{\mathbb{R}^N}  [ (a_2-1)(e^{\langle x,\mathtt t\rangle}-e^{-\langle x,\mathtt t\rangle})^2+(e^{-\langle x,\mathtt t \rangle}+e^{\langle x,\mathtt t\rangle})^2](e^{-\langle x,\mathtt t \rangle}+e^{\langle x,\mathtt t\rangle})^{a_2-2}U^{a_1+1}\\
&\quad+o(1)
\end{aligned}$$
and
$$\begin{aligned}
\bar \ell^{\frac{N-1}{2}a_2+1} e^{a_2 \bar \ell }  \mathtt D_{\mathtt n}(\bar \ell) &=-c_N^{a_2}\int_{\mathbb{R}^N}  (e^{-\langle x,\mathtt t \rangle}+e^{\langle x,\mathtt t\rangle})^{a_2-1}e^{\langle x,\mathtt t\rangle}\langle x, \mathtt n \rangle U^{a_1} \partial_{\mathtt n}  U+
o(1)\\
&=\frac{c_N^{a_2}}{a_1+1} \int_{\mathbb{R}^N}  (e^{-\langle x,\mathtt t \rangle}+e^{\langle x,\mathtt t\rangle})^{a_2-1}e^{\langle x,\mathtt t\rangle}  U^{a_1+1}+
o(1)\\ &=\frac{c_N^{a_2}}{2(a_1+1)} \int_{\mathbb{R}^N} (e^{-\langle x,\mathtt t \rangle}+e^{\langle x,\mathtt t\rangle})^{a_2}U^{a_1+1}+o(1)
\end{aligned}$$
in view of \eqref{1110} and
\begin{equation} \label{0952}
\left\langle \frac{x-\bar \ell \mathtt t}{|x-\bar \ell \mathtt t|},\mathtt t \right\rangle= -1+\mathcal O\left(\frac{|x|^2}{\bar \ell^2}\right),\quad 
\left\langle \frac{x-\bar \ell \mathtt t}{|x-\bar \ell \mathtt t|},\mathtt n \right\rangle=  \frac{\langle x, \mathtt n \rangle }{\bar \ell} +\mathcal O\left(\frac{|x|^2}{\bar \ell^2}\right).
\end{equation}
Hence, when $h \geq 1$ there holds
$$\bar \ell^{\frac{N-1}{2}a_2} e^{a_2 \bar \ell } \mathtt D {\boldsymbol\Psi}^{(1)}(z^*_{2h})=-\frac{\mathfrak c_1 \mathfrak c_3}{a_2}(\beta_{2h+1}+\beta_{2h-1}-2\beta_{2h})\mathtt t-\frac{\mathfrak c_1 \mathfrak c_4}{a_2} (\gamma_{2h+1}+\gamma_{2h-1}-2 \gamma_{2h}) \mathtt n+o(1)D_\Lambda$$
with 
$$\mathfrak c_3:=\frac{a_2 c_N^{a_2}}{2 \mathfrak c_1(a_1+1)} \int_{\mathbb{R}^N}  [ (a_2-1)(e^{\langle x,\mathtt t\rangle}-e^{-\langle x,\mathtt t\rangle})^2+(e^{-\langle x,\mathtt t \rangle}+e^{\langle x,\mathtt t\rangle})^2](e^{-\langle x,\mathtt t \rangle}+e^{\langle x,\mathtt t\rangle})^{a_2-2}U^{a_1+1}>0$$
and
$$\mathfrak c_4:=- \frac{a_2 c_N^{a_2}}{2 \mathfrak c_1(a_1+1)} \int_{\mathbb{R}^N} (e^{-\langle x,\mathtt t \rangle}+e^{\langle x,\mathtt t\rangle})^{a_2}U^{a_1+1}<0.$$
Collecting all the previous results into \eqref{taylor} we  deduce the estimates of the first component.\\

{\em The second component.}\\
If $s\in \Pi_2$ then $s=z_{2h-1}$ and $\Pi_1^s=\{z_{2h-2},z_{2h}\}$ for any $1\leq h \leq n.$ Therefore
$${\boldsymbol\Psi}^{(2)}(z_{2h-1})= \int_{\mathbb{R}^N} \Big[U(x-p_1)+U(x-p_2) \Big]^{b_1} U^{b_2}\nabla U.$$
where $p_1=z_{2h-2}-z_{2h-1}$ and $p_2=z_{2h}-z_{2h-1}$. Arguing exactly as in the previous case, we derive a Taylor expansion of ${\boldsymbol\Psi}^{(2)}$ around $z^*_{2h-1}$:
$$ {\boldsymbol\Psi}^{(2)}(z_{2h-1}) ={\boldsymbol\Psi}^{(2)} (z^*_{2h-1})-b_1  \mathtt D {\boldsymbol\Psi}^{(2)} (z^*_{2h-1})+\mathcal O(e^{-b_1 \bar \ell-\delta \bar \ell^\frac{1}{4}} +\bar \ell^\theta e^{-b_1  \bar \ell} Q_\Lambda ^\frac{\theta}{2}).$$
As in the previous paragraph (when $h\ge1$), by using the symmetries we have that 
$${\boldsymbol\Psi}^{(2)}(z^*_{2h-1})=0.$$
 It only remains to expand the derivative $ \mathtt D {\boldsymbol\Psi}^{(2)} (z^*_{2h-1})$ and we distinguish two cases.\\

If  \fbox{$h=2,\dots,n$}, since
$$p_1=
-(\bar \ell +\beta_{2h-1}-\beta_{2h-2})\mathtt t +\bar \ell (\gamma_{2h-2}-\gamma_{2h-1}) \mathtt n, \ p_2=
(\bar \ell +\beta_{2h} -\beta_{2h-1})\mathtt t +\bar \ell (\gamma_{2h}-\gamma_{2h-1}) \mathtt n, 
$$
we get
$$\bar \ell^{\frac{N-1}{2}b_1} e^{b_1 \bar \ell } \mathtt D{\boldsymbol\Psi}^{(2)}(z^*_{2h-1})=
-\frac{ \mathfrak c_5}{b_1} (\beta_{2h}+\beta_{2h-2}-2\beta_{2h-1}) \mathtt t-\frac{ \mathfrak c_6}{b_1} (\gamma_{2h}+\gamma_{2h-2}-2 \gamma_{2h-1})  \mathtt n+o(1)D_\Lambda$$
where  
$$\mathfrak c_5:=\frac{b_1 c_N^{b_1}}{2 (b_2+1)} \int_{\mathbb{R}^N}  [ (b_1-1)(e^{\langle x,\mathtt t\rangle}-e^{-\langle x,\mathtt t\rangle})^2+(e^{-\langle x,\mathtt t \rangle}+e^{\langle x,\mathtt t\rangle})^2](e^{-\langle x,\mathtt t \rangle}+e^{\langle x,\mathtt t\rangle})^{b_1-2}U^{b_2+1}>0$$
and 
$$\mathfrak c_6:=- \frac{b_1 c_N^{b_1}}{2 (b_2+1)} \int_{\mathbb{R}^N} (e^{-\langle x,\mathtt t \rangle}+e^{\langle x,\mathtt t\rangle})^{b_1}U^{b_2+1}<0.$$

If \fbox{$h=1$} then
$$p_1=
 -(\bar \ell +\beta_1+\alpha_{m+1}\sin\frac\pi k)\mathtt t -(\bar \ell \gamma_{1}-\alpha_{m+1}\cos\frac\pi k)\mathtt n   
, \quad p_2=
(\bar \ell +\beta_{2 } -\beta_{1})\mathtt t +\bar \ell (\gamma_{2 }-\gamma_{1}) \mathtt n
$$
and
$$\mathfrak d p_1= -(\beta_1+\alpha_{m+1}\sin\frac\pi k)\mathtt t -(\bar \ell \gamma_{1}-\alpha_{m+1}\cos\frac\pi k)\mathtt n, \quad
\mathfrak d p_2= ( \beta_{2 } -\beta_{1})\mathtt t +\bar \ell (\gamma_{2 }-\gamma_{1}) \mathtt n$$
in view of  $\mathtt e_1=-\sin\frac\pi k \mathtt t +\cos\frac\pi k \mathtt n$. Therefore arguing as before
$$\begin{aligned}
&\mathtt D  {\boldsymbol\Psi}^{(2)}(z^*_1)\\ &= \mathtt t \int_{\mathbb{R}^N}  [U(x+\bar \ell \mathtt t)+U(x-\bar \ell \mathtt t)]^{b_1-1}[\langle \nabla U(x+\bar \ell  \mathtt t), \mathfrak {d} p_1 \rangle+\langle \nabla U(x-\bar \ell \mathtt t), \mathfrak {d} p_2 \rangle] U^{b_2} \partial_{\mathtt t} U\\
&+\mathtt n \int_{\mathbb{R}^N}  [U(x+\bar \ell \mathtt t)+U(x-\bar \ell \mathtt t)]^{b_1-1}[\langle \nabla U(x+\bar \ell  \mathtt t), \mathfrak {d} p_1 \rangle+\langle \nabla U(x-\bar \ell \mathtt t), \mathfrak {d} p_2 \rangle] U^{b_2} \partial_{\mathtt n}U\\
& =-( \beta_1+\alpha_{m+1}\sin\frac\pi k)\mathtt t \int_{\mathbb{R}^N}  [U(x+\bar \ell \mathtt t)+U(x-\bar \ell \mathtt t)]^{b_1-1}
\partial_{\mathtt t} U(x+\bar \ell  \mathtt t) U^{b_2} \partial_{\mathtt t}  U\\
&+ ( \beta_{2 } -\beta_{1})\mathtt t \int_{\mathbb{R}^N}  [U(x+\bar \ell \mathtt t)+U(x-\bar \ell \mathtt t)]^{b_1-1} \partial_{\mathtt t} U(x-\bar \ell \mathtt t) U^{b_2} \partial_{\mathtt t} U\\
&-(\bar \ell \gamma_{1}-\alpha_{m+1}\cos\frac\pi k) \mathtt n \int_{\mathbb{R}^N}  [U(x+\bar \ell \mathtt t)+U(x-\bar \ell \mathtt t)]^{b_1-1} \partial_{\mathtt n} U(x+\bar \ell  \mathtt t)  U^{b_2} \partial_{\mathtt n} U\\
&+\bar \ell (\gamma_{2 }-\gamma_{1})   \mathtt n \int_{\mathbb{R}^N}  [U(x+\bar \ell \mathtt t)+U(x-\bar \ell \mathtt t)]^{b_1-1}\partial_{\mathtt n} U(x-\bar \ell \mathtt t)  U^{b_2} \partial_{\mathtt n} U\\
& =(\beta_{2}-2\beta_1-\alpha_{m+1}\sin\frac\pi k) \mathtt t \int_{\mathbb{R}^N}  [U(x+\bar \ell \mathtt t)+U(x-\bar \ell \mathtt t)]^{b_1-1} \partial_{\mathtt t} U(x-\bar \ell \mathtt t)U^{b_2} \partial_{\mathtt t}  U\\
&+\bar \ell (\gamma_2-2\gamma_1 +\bar\ell^{-1}\alpha_{m+1}\cos\frac\pi k)  \mathtt n \int_{\mathbb{R}^N}  [U(x+\bar \ell \mathtt t)+U(x-\bar \ell \mathtt t)]^{b_1-1}\partial_{\mathtt n} U(x-\bar \ell \mathtt t)  U^{b_2} \partial_{\mathtt n} U
\end{aligned}$$
and then
$$\bar \ell^{\frac{N-1}{2}b_1} e^{b_1 \bar \ell } \mathtt D{\boldsymbol\Psi}^{(2)}(z^*_1)=
-\frac{ \mathfrak c_5}{b_1} (\beta_2-2\beta_1 -\alpha_{m+1} \sin \frac{\pi}{k}) \mathtt t-\frac{ \mathfrak c_6}{b_1} (\gamma_2-2 \gamma_1)  \mathtt n+o(1)D_\Lambda$$
in view of  \eqref{0952}. Collecting all the previous results we  deduce the estimates of the second component.\\
\end{proof}

\subsection{The inner error}

\begin{lm}\label{new-2} 
${}$\\
As $\Lambda \to  +\infty$ the first component ${\mathscr I}_1$ satisfies
$$\begin{aligned}
\frac{\int\limits_{\mathbb R^N} {\mathscr I}_1(x) \nabla U(x-y_1) dx}{\Psi^{(0)}(\ell)}&=\left[\alpha_2-\alpha_1-2\sin\frac\pi k \alpha_1\right]\mathtt e_1 +o(1)D_\Lambda+O(e^{- \delta \ell^\frac{1}{4}}+\ell^K Q_\Lambda^\frac{1}{2})\\ 
\frac{\int\limits_{\mathbb R^N} {\mathscr I}_1(x) \nabla U(x-y_i)  dx}{\Psi^{(0)}(\ell)}&=\left[\alpha_{i+1}-2\alpha_i+\alpha_{i-1}\right]\mathtt e_1 +o(1)D_\Lambda+O(e^{- \delta \ell^\frac{1}{4}}+\ell^K Q_\Lambda^\frac{1}{2})\\
\frac{\int\limits_{\mathbb R^N} {\mathscr I}_1(x) \nabla U(x-y_{m+1})  dx}{\Psi^{(0)}(\ell)}&=
-(1+\alpha_m-\alpha_{m+1}) \mathtt e_1+o(1)D_\Lambda+O(e^{- \delta \ell^\frac{1}{4}}+\ell^K Q_\Lambda^\frac{1}{2})\\ 
\frac{ \int_{\mathbb{R}^N}{\mathscr I}_1(x) \nabla U(x -z_{2h}) dx}{\Lambda \Psi^{(1)}(\bar \ell)}&= O(e^{- \delta \bar \ell})
 \end{aligned}$$
if $i=2, \dots, m$ and $h=1, \dots, n-1$, where $K>0$ is large.  As $\Lambda \to  +\infty$ the second component ${\mathscr I}_2$ satisfies
$$ \frac{\int\limits_{\mathbb R^N}{\mathscr I}_2(x) \nabla U(x -z_{2h-1})   dx}{\Lambda \Psi^{(2)}(\bar \ell)} =
O(e^{- \delta \bar \ell})$$
if $h=1,\dots, n$. Here (see \eqref{psizero})
$$\Psi^{(0)} (t)=t^{-\frac{N-1}{2} } e^{-  t }(\mathfrak c_0+o(1))\ \hbox{as}\ t\to\infty,\ c_0>0.$$
\end{lm}
\begin{proof}

First of all,  in the Appendix we show that
\begin{equation}\label{pint1}
\int_{\mathbb{R}^N}{\mathscr I}_1(x) \nabla U(x-s) \, dx=- {\boldsymbol\Psi}^{(0)}(s) +\mathcal O(e^{-(1+\delta) \ell })
\end{equation}  
when $s \in \Pi_1$ and
\begin{equation}\label{pint2}
\int_{\mathbb{R}^N} {\mathscr I}_2(x) \nabla U(x-s) \, dx=-{\boldsymbol\Psi}^{(0)}(s)+\mathcal O( e^{-(1+\delta)2 \mu \ell\cos\frac\pi k })
\end{equation}  
when $s\in \Pi_2$, where ${\boldsymbol\Psi}^{(0)}$ is given by \eqref{in-1}. The  estimate of the leading term ${\boldsymbol\Psi}^{(0)}(s)$ is contained in \cite{musso-pacard-wei}. However, to make the paper self-contained, we can use the arguments already developed in  Lemma \ref{new-1}; in particular, we will repeatedly apply \eqref{taylor} with $a_2=1$, $a_1+1={\mathfrak p}$ and $\bar \ell$ replaced by $\ell$.\\

{{\em The  first   component.}}\\
 
If  \fbox{$s=y_1$} then $\Pi_1^s=\{y_2, R_k y_1,R_{-k}y_1\}.$ Since $p_1=y_2-y_1=\ell \mathtt e_1+(\alpha_2-\alpha_1)\mathtt e_1$ and
\begin{equation*}\begin{split}p_2=R_k y_1-y_1=(\bar\ell' +2\sin\frac\pi k \alpha_1) \mathtt t, \quad p_3=(\bar\ell'  +2\sin\frac\pi k \alpha_1) \mathtt t^*
\end{split}\end{equation*}
in view of \eqref{y1star}, $R_k \mathtt e_1-\mathtt e_1=2 \sin \frac{\pi}{k}\mathtt t$ and $R_{-k} \mathtt e_1-\mathtt e_1=2 \sin \frac{\pi}{k}\mathtt t^*$, by \eqref{1531} and \eqref{taylor} there holds
\begin{eqnarray*}
{\boldsymbol\Psi}^{(0)}(y_1)&=& - \Big[\Psi^{(0)}(\ell)\mathtt e_1+\Psi^{(0)}(\bar\ell')(\mathtt t +\mathtt t^*)\Big]
-\sum_{j=1}^3 \int_{\mathbb{R}^N}  \langle \nabla U(x-p_j^*), \mathfrak {d} p_j \rangle U^{{\mathfrak p}-1}\nabla U\\
&&+O( e^{-\ell -\delta \ell^\frac{1}{4}}+\ell e^{-\ell}Q_\Lambda^\frac{1}{2}).
\end{eqnarray*}
Arguing as in the expansions of $ \Psi^{(1)} (\bar \ell)$ and $\mathtt D _0(\bar\ell)$, thanks to \eqref{1531} we have that $\Psi^{(0)} (t)=- \int_{\mathbb{R}^N} U(x-t \mathtt e_1) U^{{\mathfrak p}-1} \partial_1 U$ can be expanded as in \eqref{psizero} with
\begin{equation}\label{1328}
{\mathfrak c}_0=\frac{c_N}{{\mathfrak p}}  \int_{\mathbb{R}^N}  e^{x_1 } U^{{\mathfrak p}}>0
\end{equation}
and, through an integration by parts and a rotation,
\begin{eqnarray}
&&t^{\frac{N-1}{2}}e^{t} \int_{\mathbb{R}^N}  \langle \nabla U(x-t e), \mathfrak {d} p \rangle U^{{\mathfrak p}-1}\nabla U=
c_N \langle e, \mathfrak {d} p \rangle \int_{\mathbb{R}^N}  e^{\langle x, e \rangle }  U^{{\mathfrak p}-1}\nabla U
+o(1)D_t \nonumber \\
&&=- {\mathfrak c}_0 \langle e, \mathfrak {d} p \rangle e+o(1)D_t 
=-  t^{\frac{N-1}{2}}e^{t} \Psi^{(0)} (t) \langle e, \mathfrak {d} p \rangle e+o(1)D_t \label{1509}
\end{eqnarray}
for $|e|=1$.  Since
$$\Psi^{(0)}(\ell)\mathtt e_1+\Psi^{(0)}(\bar\ell')(\mathtt t+\mathtt t^*)=
\Psi^{(0)}(\ell)-2\Psi^{(0)}(\bar\ell')\sin\frac\pi k =0$$
thanks to the balance condition \eqref{ba1}, by \eqref{1509} we deduce that
$$\begin{aligned}
&{\boldsymbol\Psi}^{(0)}(y_1)\\ & 
=-  \Psi^{(0)} (\ell) \Big[  (\alpha_2-\alpha_1)\mathtt e_1+\frac{2 \Psi^{(0)}(\bar\ell') \sin \frac{\pi}{k}}{\Psi^{(0)}(\ell)}\alpha_1 (\mathtt t+\mathtt t^*)
+o(1)D_\Lambda +O( e^{ -\delta \ell^\frac{1}{4}}+\ell^K Q_\Lambda^\frac{1}{2})\Big]\\
&= -  \Psi^{(0)} (\ell) \Big[  (\alpha_2-\alpha_1-2\sin \frac{\pi}{k}\alpha_1)\mathtt e_1 +o(1)D_\Lambda +O( e^{ -\delta \ell^\frac{1}{4}}+\ell^K Q_\Lambda^\frac{1}{2})\Big].
\end{aligned}$$

If    \fbox{$s=y_i,\ i=2,\dots,m$}, then $ \Pi_1^s=\{ y_{i-1},y_{i+1}\}$ and
$$p_1=y_{i-1}-y_i=-(\ell +\alpha_{i}-\alpha_{i-1}) \mathtt e_1,\quad p_2=
y_{i+1}-y_i=(\ell +\alpha_{i+1}-\alpha_i)\mathtt e_1.$$
Since ${\boldsymbol\Psi}^{(0)}(y_1^*)= - \Psi^{(0)}(\ell)[\mathtt e_1-\mathtt e_1]=0$, by \eqref{1509} we deduce that
\begin{eqnarray*}
{\boldsymbol\Psi}^{(0)}(y_i)&=& 
-  \Psi^{(0)} (\ell)\Big[(\alpha_{i+1}-2\alpha_i+\alpha_{i-1})\mathtt e_1+o(1)D_\Lambda +O( e^{ -\delta \ell^\frac{1}{4}}+\ell^K Q_\Lambda^\frac{1}{2})\Big].
\end{eqnarray*}

If  \fbox{$s=y_{m+1}$} then $\Pi_1^s=\{y_m\}$ and $p_1=y_m-y_{m+1}=-(\ell +\alpha_{m+1}-\alpha_m)\mathtt e_1$. Since ${\boldsymbol\Psi}^{(0)}(y_{m+1}^*)=\Psi^{(0)}(\ell)\mathtt e_1$,  by \eqref{1509} we deduce that
\begin{eqnarray*}
{\boldsymbol\Psi}^{(0)}(y_{m+1})&=&  \Psi^{(0)} (\ell)\Big[(1 +\alpha_m-\alpha_{m+1})\mathtt e_1+o(1)D_\Lambda +O( e^{ -\delta \ell^\frac{1}{4}}+\ell^K Q_\Lambda^\frac{1}{2})\Big].
\end{eqnarray*}

Finally, if  \fbox{$s=z_{2h},\ h=1,\dots,n-1$} then $\Pi_1^s=\{ z_{2h-2},z_{2h+2}\}$ and
\begin{eqnarray*}
{\boldsymbol\Psi}^{(0)}(z_{2h})=\mathcal O(e^{-2\bar \ell})=\mathcal O (\Lambda \Psi^{(1)}(\bar \ell)e^{-\delta \bar \ell})
\end{eqnarray*}
in view of $\mu >\frac{1}{2}$.\\

{ {\em The  second   component.}}\\
Finally, if \fbox{$s=z_{2h-1}, h=1, \dots, n$} then $\Pi_2^s=\{z_3,R_{-k} z_{2n-1}\}$ when $h=1$ and $\Pi_2^s=\{z_{2h-3},z_{2h+1}\}$ when $h=2, \dots, n$. Therefore
\begin{eqnarray*}
{\boldsymbol\Psi}^{(0)}(z_{2h-1})=\mathcal O(e^{-2\mu \cos \frac{\pi}{k} \ell})=\mathcal O (\Lambda \Psi^{(2)}(\bar \ell)e^{-\delta \bar \ell})
\end{eqnarray*}
in view of \eqref{1217} and $\mu>\frac 1{a-b+2\cos\frac\pi k}$.\\
\end{proof}
 
\section{Proof of main Theorem completed} \label{secmain}
In this section we complete the proof of Theorem \ref{main} by solving \eqref{sys}. It is enough to show that, if $\Lambda$ is large enough, we can find  the $(2n+m)$ parameters $\boldsymbol\alpha:=(\alpha_1,\dots,\alpha_{m+1})\in\mathbb R^{m+1}$, $\boldsymbol\beta:=(\beta_1,\dots,\beta_{n-1})\in\mathbb R^{n-1}$ and $\boldsymbol\gamma:=(\gamma_1,\dots,\gamma_n)\in\mathbb R^{n}$ as in \eqref{cpabg}-\eqref{symme} such that the term $\mathscr M({\bf h})$ in \eqref{pro-pro} (see also \eqref{MMM}) is zero.  To estimate the   contribution coming from the linear term $\mathscr L(\phi)$ in \eqref{L} and the quadratic one $\mathscr Q(\phi)$ in\eqref{Q}, we have the following result.
\begin{lm} \label{linear-rate}
 Let $i=1,2$ and $p\in \Pi_i$. For any $j=1,\dots,n$ there hold  
$$\int_{\mathbb{R}^N} \mathscr L_i(\phi) \partial_j U(x-p) dx, \ \int_{\mathbb{R}^N} \mathscr Q_i(\phi) \partial_j U(x-p) dx=   
 \mathcal O\left(\Lambda \Psi^{(i)}(\bar\ell)e^{-\delta\bar\ell}\right)$$
for some $\delta>0$, provided ${\mathfrak p}>\frac{3}{2}$  and $\mu$ satisfies \eqref{d5} and \eqref{1642}-\eqref{1653}.
\end{lm}
%
\begin{proof}
 Since  $(-\Delta +1-{\mathfrak p} U^{{\mathfrak p}-1})(\partial_j U) =0 $ in $\mathbb{R}^N$, through an integration by parts notice that
$$\begin{aligned}
& \int_{\mathbb{R}^N} \mathscr L_i(\phi) \partial_j U(x-p) dx=
\int_{\mathbb{R}^N} \Big({\mathfrak p} U^{{\mathfrak p}-1} (x-p)-{\mathfrak p} U_i^{{\mathfrak p}-1}\Big)\phi_i \partial_j U (x-p)dx\\
&+\Lambda \cdot \left\{\begin{array}{ll}
a_1  \int_{\mathbb{R}^N} U_1^{a_1-1}U_2^{a_2}\phi_1 \partial_j U (x-p)dx
 +a_2  \int_{\mathbb{R}^N}  U_1^{a_1}U_2^{a_2-1}\phi_2 \partial_j U (x-p)dx &\hbox{if }i=1\\
b_2  \int_{\mathbb{R}^N}  U_1^{b_1}U_2^{b_2-1}\phi_2 \partial_j U (x-p)dx+
b_1  \int_{\mathbb{R}^N}  U_1^{b_1-1}U_2^{b_2}\phi_1 \partial_j U (x-p)dx&\hbox{if }i=2. \end{array} \right.
\end{aligned}$$ 
Since 
$$|U^{{\mathfrak p}-1} (x-p)-U_i^{{\mathfrak p}-1}|=\mathcal O\Big(
\sum_{q\in \Pi_i \setminus\{p\} }U^{{\mathfrak p}-1}(x-q)+ \underbrace{U^{{\mathfrak p}-2}(x-p) \sum_{q\in \Pi_i \setminus \{p\}}U(x-q)}_{\hbox{if }{\mathfrak p}\geq 2}\Big),$$
we have that
$$\int_{\mathbb{R}^N} \mathscr L_i(\phi) \partial_j U(x-p) dx
=\mathcal O\Big(  e^{-\min\{{\mathfrak p}-1,1\}\rho_i}+\underbrace{\Lambda e^{-(a-1)\rho}}_{\hbox{if }i=1}
+\underbrace{\Lambda e^{-(b-1)\rho}}_{\hbox{if }i=2}\Big)  \| \phi\|_* $$
in view of \eqref{1217}, \eqref{key} and $U_i^\gamma |\partial_j U(x-p)|=\mathcal O\Big(\displaystyle \sum_{q \in \Pi_i} U^{\gamma+1}(x-q) \Big)$. By \eqref{expansion}, \eqref{1217}  and Proposition \ref{nonlinear} we have that for some $\delta>0$ there hold:
\begin{itemize} \item $\mathcal O\Big(\underbrace{\Lambda e^{-(a-1)\rho}}_{\hbox{if }i=1}
+\underbrace{\Lambda e^{-(b-1)\rho}}_{\hbox{if }i=2}\Big)  \| \phi\|_* = \mathcal O\left(\Lambda \Psi^{(i)}(\bar\ell)e^{-\delta\bar\ell}\right)$ provided $\mu<\frac{1}{a+1-c}$ and 
\begin{equation} \label{1642}
 \mu<\min\{{\mathfrak p}-1,1\};
\end{equation}
\item $\mathcal O\Big(e^{-\min\{{\mathfrak p}-1,1\}\rho_1} \Big) \| \phi\|_* =\mathcal O(\Lambda \Psi^{(1)}(\bar\ell)e^{-\delta\bar\ell})$ provided ${\mathfrak p}>\frac{3}{2}$ and $\mu$ satisfies \eqref{1642} in view of $a-c<1$;
\item $\mathcal O\Big( e^{-\min\{{\mathfrak p}-1,1\}\rho_2} \Big) \| \phi\|_* =\mathcal O(\Lambda \Psi^{(2)}(\bar\ell)e^{-\delta\bar\ell})$ provided ${\mathfrak p}>\frac{3}{2}$  and $\mu$ satisfies
\begin{equation} \label{1653}
(b-a)\Big(1-\min\{{\mathfrak p}-1,1\}\Big)\mu< 2\min\{{\mathfrak p}-1,1\}-1
\end{equation}
in view of $2\mu \cos \frac{\pi}{k}>\mu(b-a)+1$.\\
\end{itemize}

Concerning $\mathscr Q$, we first consider the case $i=1$.  Since $\Lambda e^{-(c-1)\rho}\to 0$ and $\Lambda e^{-(a_2-2)\rho}\to +\infty$ as $\Lambda \to +\infty$ in view of \eqref{expansion}, \eqref{1217} and \eqref{d5}, by \eqref{Q1}-\eqref{Q1pp} we get that
$$\begin{aligned}
&\int_{\mathbb{R}^N}\mathscr Q_1(\phi)\partial_j U(x-p)\\
&=\mathcal O\Big(\|\phi\|_*^2+\Lambda e^{-\rho}\|\phi\|_*^{a_1}+\Lambda \|\phi\|_*^{a_2}+\Lambda e^{-(c-1)\rho}\|\phi\|_*^2+\Lambda e^{-(a_2-2)\rho}\|\phi\|_*^2\Big)\\&=\mathcal O\Big(\Lambda e^{-\rho}\|\phi\|_*^{a_1}+\Lambda \|\phi\|_*^{a_2}+\Lambda e^{-(a_2-2)\rho}\|\phi\|_*^2\Big).
\end{aligned}$$
By \eqref{expansion}, \eqref{1217}  and Proposition \ref{nonlinear} we have that 
\begin{eqnarray*}
\int_{\mathbb{R}^N}\mathscr Q_1(\phi)\partial_j U(x-p)=\mathcal O(\Lambda \Psi^{(1)}(\bar\ell)e^{-\delta\bar\ell})
\end{eqnarray*}
does hold for some $\delta>0$ provided  $\mu<\frac{1}{a+1-c}$ and $\mu$ satisfies \eqref{1642} (in view of $a<a_1+1$).
In a similar way,  we argue for the estimate of $\int_{\mathbb{R}^N}\mathscr Q_2(\phi)\partial_j U(x-p)$ with $p \in \Pi_2$.
\end{proof}

\subsection{Solving the reduced problem}
Since $\{\int_{\mathbb{R}^N} \mathscr M_i({\bf h})\nabla U(x-s)\ dx\}_{i=1,2, \ s \in \Pi_i}$,  provides an expression in ${\bf h}$ which is almost diagonal,  to have ${\bf h}=0$ it's enough to impose $\int_{\mathbb{R}^N} [\mathscr L_i(\phi)+\mathscr E_i+\mathscr Q_i(\phi)]\nabla U(x-s)\ dx=0$ for any $i=1,2$ and $s \in \Pi_i$ in view of \eqref{pro-pro}. 
By symmetry arguments and the expansions in Lemmas \ref{new-1},  \ref{new-2} and \ref{linear-rate}
we are lead to find
$(\boldsymbol\alpha,\boldsymbol\beta,\boldsymbol\gamma)\to 0$ as $\Lambda\to\infty$  which solve the non-linear system
\begin{equation}\label{fine1}\mathtt L(\boldsymbol\alpha,\boldsymbol\beta,\boldsymbol\gamma)=\mathtt E_\Lambda +\mathtt Q_\Lambda(\boldsymbol\alpha,\boldsymbol\beta,\boldsymbol\gamma)\end{equation}
where the linear operator $\mathtt L :\mathbb R^{m+1}\times\mathbb R^{n-1}\times\mathbb R^{n}\to \mathbb R^{m+1}\times\mathbb R^{n-1}\times\mathbb R^{n}$ is a small perturbation of
$$\mathtt L_0 (\boldsymbol\alpha,\boldsymbol\beta,\boldsymbol\gamma)=
\left(
\begin{aligned}
&  \alpha_2-2\alpha_1+\left(1 -2 \sin\frac\pi k\right) \alpha_1   \\
&\quad\vdots\\
& \alpha_{i+1}-2\alpha_i+\alpha_{i-1}\\
&\quad\vdots\\
&  \alpha_{m}-2\alpha_{m+1}+a_2 \beta_1+\left(1+a_2\sin\frac\pi k \right)\alpha_{m+1} -\mathfrak c_2 \gamma_1 \\
& \beta_2-2\beta_1-\sin\frac\pi k \alpha_{m+1} \\
&\quad\vdots\\
& \beta_{h+1}- 2\beta_{h}+\beta_{h-1} \\
&\quad\vdots\\
&  \beta_{n-2}-2\beta_{n-1} \\
&\gamma_2-2\gamma_1 \\
&\quad\vdots\\ &\gamma_{h+1}-2\gamma_{h}+\gamma_{h-1}  \\
&\quad\vdots\\
&2 \gamma_{n-1}-2\gamma_{n} \\
\end{aligned}\right).$$
in view of  \eqref{symme} and \eqref{ba2}. The error term $\mathtt E_\Lambda:\mathbb R^{m+1}\times\mathbb R^{n-1}\times\mathbb R^{n}\to \mathbb R^{m+1}\times\mathbb R^{n-1}\times\mathbb R^{n}$ is a  continuous function in $(\boldsymbol\alpha,\boldsymbol\beta,\boldsymbol\gamma)$ so that $\mathtt E_\Lambda(\boldsymbol\alpha,\boldsymbol\beta,\boldsymbol\gamma)=\mathcal O(e^{-\delta \ell^\frac{1}{4}})$ as 
$ \Lambda\to\infty$, uniformly with respect to  $(\boldsymbol\alpha,\boldsymbol\beta,\boldsymbol\gamma)$ in bounded sets.
The non-linear term $\mathtt Q_\Lambda:\mathbb R^{m+1}\times\mathbb R^{n-1}\times\mathbb R^{n}\to \mathbb R^{m+1}\times\mathbb R^{n-1}\times\mathbb R^{n}$ is a continuous function in $(\boldsymbol\alpha,\boldsymbol\beta,\boldsymbol\gamma)$ so that 
$$\mathtt Q_\Lambda(\boldsymbol\alpha,\boldsymbol\beta,\boldsymbol\gamma)=\mathcal O\Big(\ell^K [|\alpha| +|\beta| +|\gamma| ]^{ \theta}\Big)$$
for some $1<\theta\leq 2$, uniformly with respect to  $(\boldsymbol\alpha,\boldsymbol\beta,\boldsymbol\gamma)$ in bounded sets. \\

If the operator $\mathtt L_\Lambda$  is invertible,  we can use  a simple contraction argument to solve \eqref{fine1} with $(\boldsymbol\alpha,\boldsymbol\beta,\boldsymbol\gamma)=\mathcal O(e^{-\delta \ell^\frac{1}{4}})$.  Since 
$\mathtt L_\Lambda=\mathtt L_0+o(1)$,  it is therefore enough to prove that the linear problem
$$ \left\{\begin{aligned}
&  \alpha_2-2\alpha_1+\left(1 -2 \sin\frac\pi k\right) \alpha_1 =0 \\
& \alpha_{i+1}-2\alpha_i+\alpha_{i-1}=0,\ 
 \hbox{if}\ i=2, \dots, m \\
& \alpha_{m}-2\alpha_{m+1}+a_2 \beta_1+\left(1+a_2\sin\frac\pi k \right)\alpha_{m+1} -\mathfrak c_2 \gamma_1=0\\
& \beta_2-2\beta_1-\sin\frac\pi k \alpha_{m+1}=0\\
& \beta_{h+1}- 2\beta_{h}+\beta_{h-1}=0,\   \hbox{if}\ h=2, \dots, n-2 \\
&   \beta_{n-2}-2\beta_{n-1}=0\\
&\gamma_2-2\gamma_1=0 \\
&\gamma_{h+1}-2\gamma_{h}+\gamma_{h-1}=0 ,\    \hbox{if}\ h=2, \dots, n-1 \\
& 2\gamma_{n-1}-2\gamma_{n}=0
\end{aligned}\right.$$
has only the trivial solution. Since $\gamma_{n}-\gamma_{n-1}=0$ and $\gamma_{h+1}-\gamma_h=\gamma_h -\gamma_{h-1}$ for $h=2,\dots,n-1$,  we have that $\gamma_1=\gamma_2-\gamma_1=0$ and then $\gamma_i=0$ for any $i=1,\dots,n.$  Let us solve in $\alpha_i$ and $\beta_i$. 
 Given $M\ge2$ we introduce the $M\times M$ matrix
\begin{equation}\label{tn}T_M:=\left(\begin{matrix}
&-2&1&0&\cdots&0\\
&1&-2&1&\cdots&0\\
&0&1&\ddots&\vdots&0\\
&\vdots&\vdots&1&-2&1\\
&0&\cdots&0&1&-2\\
\end{matrix}\right).\end{equation}
It is easy to check that
\begin{equation}\label{tnt}T^{-1}_M\left(\begin{matrix}
1\\
 0\\
\vdots \\
0 \\
0 \\
\end{matrix}\right)=-\frac1{M+1}\left(\begin{matrix}
M\\
 M-1\\
\vdots \\
2 \\
1 \\
\end{matrix}\right)\quad \hbox{and}\quad T^{-1}_M\left(\begin{matrix}
0\\
 0\\
\vdots \\
0 \\
1 \\
\end{matrix}\right)=-\frac1{M+1}\left(\begin{matrix}
1\\
 2\\
\vdots \\
M-1 \\
M \\
\end{matrix}\right)\end{equation}Using \eqref{tn}  we can    rewrite the equations as
$$T_{m+1}\left(\begin{matrix}
\alpha_1\\
\vdots \\
\alpha_{m+1} \\
\end{matrix}\right)=-\left(\begin{matrix}
\left(1 -2 \sin\frac\pi k\right) \alpha_1\\
0\\
\vdots \\
0\\
a_2 \beta_1+\left(1+a_2\sin\frac\pi k \right)\alpha_{m+1} \\
\end{matrix}\right)\ \hbox{and}\ T_{n-1}\left(\begin{matrix}
\beta_1\\
\vdots \\
\beta_{n-1} \\
\end{matrix}\right)=\left(\begin{matrix}
\sin\frac\pi k \alpha_{m+1}\\
0\\
\vdots \\
0\\
0\\
\end{matrix}\right)$$
 
By \eqref{tnt} applied to the matrix $T_{n-1}$ it follows that
\begin{equation} \label{183511}
\beta_1=-\frac{n-1}n\sin\frac\pi k \alpha_{m+1}
\end{equation}
and then,  by \eqref{tnt} applied to the matrix $T_{m+1}$,  we get that 
$$\left\{\begin{aligned}&\left(1-2\sin\frac\pi k\right)(m+1)\alpha_1+\left(1+\frac{a_2}n\sin\frac\pi k \right)\alpha_{m+1}=(m+2)\alpha_1\\
&\left(1-2\sin\frac\pi k\right)\alpha_1+\left(1+\frac{a_2}n\sin\frac\pi k \right)(m+1)\alpha_{m+1}=(m+2)\alpha_{m+1}\\\end{aligned}\right.$$
in view of \eqref{183511}. Since the matrix
$$ A:=\left(\begin{matrix}
\left(1-2\sin\frac\pi k\right)(m+1)-(m+2)&\left(1+\frac{a_2}n\sin\frac\pi k \right)\\
 & \\
 \left(1-2\sin\frac\pi k\right)&\left(1+\frac{a_2}n\sin\frac\pi k \right)(m+1)-(m+2)\\
\end{matrix}\right) 
$$
satisfies
$$\hbox{det }A=(m+2) \sin \frac{\pi}{k}\Big[2-\frac{a_2}{n}(2m \sin \frac{\pi}{k}+1)\Big]=
2(m+2) \sin \frac{\pi}{k}(1- a \mu )$$
in view of \eqref{0849} and \eqref{drate},  we deduce that the matrix $A$ is invertible thanks to \eqref{d5} and then $\alpha_1=\alpha_{m+1}=0$. Once $\alpha_{m+1}=0$, we have  that $-\beta_{n-1}=\beta_{n-1}-\beta_{n-2}=\dots=\beta_2-\beta_1=\beta_1$ and then $\beta_h=-h \beta_{n-1}$ for all $h=1,\dots,n-1$, which implies $\beta_1=\dots=\beta_{n-1}=0$.  Since $\alpha_1=\alpha_{m+1}=0$, we also easily deduce that $\alpha_1=\dots=\alpha_{m+1}=0$.

\appendix \section{Proof of (\ref{us1})-(\ref{us2}) and (\ref{pint1})-(\ref{pint2}) }\label{app1}
Let $\gamma>1$.  Since
$$(1+t)^{\gamma}=1+t^\gamma+\mathcal O(t+t^{\gamma-1})=1+\gamma t+t^\gamma+\mathcal O(t^{\alpha_1}+t^{\alpha_2})$$ 
holds uniformly in $t \geq 0$ and for suitable $1<\alpha_1 \leq \alpha_2<\gamma$ (with $\alpha_1 \leq 2$ and $\alpha_2 \geq \gamma-1$ if $\gamma\geq 2$),
we deduce the validity of 
$$(t_1+t_2)^{\gamma}=t_1^{\gamma}+t_2^{\gamma}+\mathcal O(t_1^{\gamma-1} t_2+t_1 t_2^{\gamma-1})=t_1^{\gamma}+\gamma t_1^{\gamma-1}t_2+t_2^{\gamma}+\mathcal O(t_1^{\gamma-\alpha_1}t_2^{\alpha_1}+t_1^{\gamma-\alpha_2} t_2^{\alpha_2})$$  
uniformly for $t_1,t_2\geq 0$, which by iteration implies
\begin{eqnarray} \label{710}
(\sum_i t_i)^{\gamma}=\sum_i t_i^{\gamma}+\mathcal  O\Big(\sum_{i\not=j}t_i^{\gamma-1}t_j\Big)  
\end{eqnarray}
and
\begin{eqnarray} 
&&(\sum_i t_i)^{\gamma}=
t_1^{\gamma}+\gamma t_1^{\gamma-1} \sum_{i\not=1} t_i+(\sum_{i\not=1} t_i)^{\gamma}+\mathcal  O\Big(t_1^{\gamma-\alpha_1} (\sum_{i\not=1} t_i)^{\alpha_1} +t_1^{\gamma-\alpha_2} (\sum_{i\not=1} t_i)^{\alpha_2}\Big) \nonumber \\
&&=\sum_i t_i^{\gamma}+\gamma t_1^{\gamma-1} \sum_{i\not=1} t_i+\mathcal  O\Big( \sum_{j=1}^2  \sum_{i \not=1}  t_1^{\gamma-\alpha_j}  t_i^{\alpha _j}+\sum_{i, j \not=1 \atop i\not=j}t_i^{\gamma-1}t_j\Big) 
\label{711}
\end{eqnarray}
uniformly for nonnegative $t_i$'s in view of \eqref{710}.\\

First, let us prove \eqref{us1}-\eqref{us2}.  Given $s\in\Pi_1$, since
$$U_1^{a_1}=U^{a_1}(x-s)+\mathcal O\Big(U^{a_1-1}(x-s)\sum_{p\in\Pi_1\atop p\neq s} U(x-p)+\sum_{p\in\Pi_1\atop p\neq s} U^{a_1}(x-p)\Big)$$ 
and 
$$\begin{aligned}U_2^{a_2}&=\Big(\sum_{q \in \Pi_2^s} U(x-q)\Big)^{a_2}+\mathcal O\Big(\sum_{q \in \Pi_2 \setminus \Pi_2^s} U^{a_2}(x-q)\Big)\\ 
&+\mathcal O\Big(\sum_{q \in \Pi_2^s,t \in \Pi_2 \setminus \Pi_2^s}U^{a_2-1}(x-q)  U(x-t)+U(x-q)  U^{a_2-1}(x-t) \Big)\end{aligned}$$
in view of \eqref{710}, by $|\nabla U|\leq c U$ we have that
\begin{eqnarray}
&&\int_{\mathbb{R}^N} {\mathscr O}_1(x)\nabla U(x -s) dx =\Lambda {\boldsymbol\Psi}^{(1)}(s)
+ \mathcal O\Big(\Lambda \sum_{q \in\Pi_2 \setminus \Pi_2^s}\int_{\mathbb R^N}U^{a_1+1}(x-s)U^{a_2}(x-q)\Big) 
 \nonumber\\
&&+ \mathcal O\Big(\Lambda \sum_{q \in \Pi_2^s,t \in \Pi_2 \setminus \Pi_2^s}\int_{\mathbb R^N}U^{a_1+1}(x-s)[U^{a_2-1}(x-q)U(x-t)+U(x-q)U^{a_2-1}(x-t)]\Big) \nonumber\\
&&+\mathcal O\Big(\Lambda \sum_{p\in\Pi_1\atop p\neq s}\sum_{q\in\Pi_2}\int_{\mathbb R^N}[U^{a_1}(x-s)U(x-p)+U(x-s) U^{a_1}(x-p)]U^{a_2}(x-q)\Big) \label{1702}
\end{eqnarray} 
where ${\mathscr O}_1$ is given by \eqref{1026}.  Since \eqref{1217} implies
\begin{equation} \label{1710}
|r-q| \geq \left\{ \begin{array}{ll}\bar \ell+\mathcal O(1)&\hbox{if }q \in \Pi_2\\
(1+\delta) \bar \ell & \hbox{if }q \in \Pi_2 \setminus \Pi_2^r \end{array}\right., \
|s-p| \geq \left\{ \begin{array}{ll}(1+\delta) \bar \ell&\hbox{if }s,p \in \Pi_1^q\\
\bar \ell +\mathcal O(1)& \hbox{otherwise}\end{array}\right.
\end{equation}
for some $\delta>0$ when $r,s,p \in \Pi_1$, $s \not=p$ and $q \in \Pi_2$, by \eqref{0849} and \eqref{key} we have that 
\begin{eqnarray}
&& \int_{\mathbb R^N} U^{a_1+1}(x-s) U^{a_2}(x-q)\, dx=\mathcal O(e^{-a_2|s-q|}) \nonumber \\
& &=\mathcal O(e^{-(1+\delta)a_2 \bar \ell })\hbox{ if }q  \in \Pi_2 \setminus \Pi_2^s \label{1659}\\
&& \int_{\mathbb R^N} U^{a_1+1}(x-s)[U^{a_2-1}(x-q)U(x-t)+U(x-q)U^{a_2-1}(x-t)]\, dx \nonumber \\
&& =\mathcal O(e^{-(a_2-1)|s-q|-|s-t|}+e^{-(a_2-1)|s-t|-|s-q|})\nonumber \\
&& =\mathcal O(e^{-(1+\delta)a_2 \bar \ell })\hbox{ if }q\in \Pi_2^s, t  \in \Pi_2 \setminus \Pi_2^s \label{1700}\\
&&  \int_{\mathbb R^N}[U^{a_1}(x-s) U(x-p)+U(x-s) U^{a_1}(x-p)]U^{a_2}(x-q)\, dx \nonumber\\
&& =\mathcal O(e^{-(a-1)|s-q|-\delta_0|s-p|-(1-\delta_0)|p-q|}
+e^{-(a-1)|p-q|-\delta_0|s-p|-(1-\delta_0)|s-q|})\nonumber \\&& =\mathcal O(e^{-(1+\delta)a_2 \bar \ell })\label{1701}
\end{eqnarray}
hold for some $\delta>0$, where $0<\delta_0<\min\{1, a_1+1-a_2\}$.  In the last estimate,  notice we have used that $s \in \Pi_1 \setminus \Pi_1^q$ is equivalent to $q \in \Pi_2 \setminus \Pi_2^s$. Inserting \eqref{1659}-\eqref{1701} into \eqref{1702} we deduce the validity of  \eqref{us1}. In a similar way, replacing \eqref{1710} with the property
$$|r-q| \geq \left\{ \begin{array}{ll}\bar \ell+\mathcal O(1)&\hbox{if }q \in \Pi_1\\
(1+\delta) \bar \ell & \hbox{if }q \in \Pi_1 \setminus \Pi_1^r \end{array}\right., \
|s-p| \geq \left\{ \begin{array}{ll}(1+\delta) \bar \ell&\hbox{if }s,p \in \Pi_2^q\\
\bar \ell +\mathcal O(1)& \hbox{otherwise}\end{array}\right.$$
for some $\delta>0$ when $q \in \Pi_1$ and $r,s,p \in \Pi_2$, $s \not=p$, we can establish the validity of  \eqref{us2} too. \\

Now,  let us establish (\ref{pint1})-(\ref{pint2}).  Since $|\nabla U|\leq c U$, we can use \eqref{711} with $t_1=U(x-s)$, $\{t_i\}_{i \not=1}=\{U(x-p)\}_{p \in \Pi_1, p\not=s}$, $\gamma={\mathfrak p}$ to get
\begin{eqnarray}
&&\int_{\mathbb{R}^N} {\mathscr I}_1 \nabla U(x-s) \, dx=-\int_{\mathbb{R}^N}\left[\Big(\sum_{p\in\Pi_1}U(x-p)\Big)^{{\mathfrak p}}-\sum_{p\in\Pi_1}U^{{\mathfrak p}}(x-p)\right] \nabla U(x-s)\, dx \nonumber\\
&&=-\frac{1}{{\mathfrak p}}\sum_{p\in\Pi_1\atop p\neq s} \int_{\mathbb R^N} U(x-p) \nabla U^{{\mathfrak p}}(x-s)+\mathcal  O\Big(\sum_{j=1}^2 \sum_{p\in\Pi_1\atop p\neq s}\int_{\mathbb R^N} U^{{\mathfrak p}+1-\alpha_j}(x-s)U^{\alpha_j} (x-p)\, dx\Big)\nonumber \\
&&+\mathcal O\Big(\sum_{p,q\in\Pi_1 \setminus \{s\}\atop p\not=q}\int_{\mathbb R^N} U(x-s) U^{{\mathfrak p}-1}(x-p)U(x-q)\, dx\Big), \label{1736}
\end{eqnarray} 
where ${\mathscr I}_1$ is given by \eqref{1026}. Since $\Pi_1$ cannot contain three distinct points with all mutual distances  of minimal order $\ell+O(1)$ (see \eqref{1217}), by \eqref{1623}-\eqref{1117} we have that,  if $p \in \Pi_1 \setminus \{s\}$ and $q\in\Pi_1^s$, there hold
\begin{equation} \label{1746}|p-s| \geq \left\{ \begin{array}{ll}\ell+\mathcal O(1)&\hbox{if }p\in \Pi_1\\
(1+\delta) \ell & \hbox{if }p\notin \Pi_1^s \end{array}\right., \
\max\{|p-s|,|p-q| \} \geq (1+\delta)\ell \hbox{ for }p \not= q
\end{equation} 
for some $\delta>0$. By \eqref{key} and \eqref{1746} we deduce that for some $\delta>0$ there hold
\begin{eqnarray}
&&\hspace{-1.7cm} \int_{\mathbb R^N} U(x-p) \nabla U^{{\mathfrak p}}(x-s)\, dx=\mathcal O(e^{-|p-s|})=\mathcal O(e^{-(1+\delta)\ell })\hbox{ if }p  \in \Pi_1 \setminus \Pi_1^s \label{1800}\\
&&\hspace{-1.7cm} \int_{\mathbb R^N} U^{{\mathfrak p}+1-\alpha_j}(x-s)U^{\alpha_j}(x-p)\, dx
=\mathcal O(e^{-(1+\delta)|s-p|})=\mathcal O(e^{-(1+\delta)\ell })\label{1801} \\
&&\hspace{-1.7cm} \int_{\mathbb R^N} U(x-s) U^{{\mathfrak p}-1}(x-p)U(x-q)\, dx=\mathcal O(e^{-|s-q|})=\mathcal O(e^{-(1+\delta) \ell})\hbox{ if }q \in \Pi_1 \setminus \Pi_1^s\label{1802}
\end{eqnarray}
thanks to $\min\{{\mathfrak p}+1-\alpha_j,\alpha_j\}>1$ and
\begin{eqnarray}
\int_{\mathbb R^N} U(x-s) U^{{\mathfrak p}-1}(x-p)U(x-q)\, dx&=&\mathcal O(e^{-\delta_0 \max\{|p-s|,|p-q|\}
-(1-\delta_0)|s-q|})\nonumber \\
&=& \mathcal O(e^{-(1+\delta) \ell})
\label{1953}
\end{eqnarray}
if $q \in \Pi_1^s$ and $p\not= q,s$, where $0<\delta_0<\min\{1,\frac{{\mathfrak p}-1}{2}\}$. Inserting \eqref{1800}-\eqref{1953} into \eqref{1736} we deduce the validity of \eqref{pint1}. Since \eqref{1746} is valid for $\Pi_2$ with $\ell$ replaced by the minimal distance $2\mu \cos \frac{\pi}{k} \ell$ in $\Pi_2$ (see \eqref{1217}),  in a similar way we can establish the validity of \eqref{pint2}.

\section*{Acknowledgments}
Pierpaolo Esposito acknowledges support by INDAM-GNAMPA and by the project PNRR-M4C2-I1.1-PRIN 2022-PE1-Variational and Analytical aspects of Geometric PDEs-F53D23002690006 - Funded by the U.E.-NextGenerationEU.
 Angela Pistoia acknowledges support of INDAM-GNAMPA project “Problemi di doppia curvatura su variet\`a a bordo e legami con le EDP di tipo ellittico” and of the project “Pattern formation in nonlinear phenomena” funded by the MUR Progetti di Ricerca di Rilevante Interesse Nazionale (PRIN) Bando 2022 grant 20227HX33Z. 
Giusi Vaira acknowledges support of 
the MUR-PRIN-P2022YFAJH ``Linear and Nonlinear PDE’s: New directions and Applications" and the support of the INdAM-GNAMPA project ``Fenomeni non lineari: problemi locali e non locali e loro applicazioni" CUP E5324001950001.
Pablo Figueroa acknowledges support by ANID (Chile) through the project FONDECYT 1201884. Part of this work was carried out while P. F. was visiting the Dipartimento di Matematica e Fisica, Universit\`a degli Studi Roma Tre, Roma, Valdivia. He would like to thank Professors Pierpaolo Esposito and Angela Pistoia for their warming hospitality and support.
\\

\bibliography{system.bib}
\bibliographystyle{abbrv}
 \end{document}